\documentclass[Unicode]{cedram-alco}
% ...

\usepackage{enumitem}
\usepackage{graphicx}
\usepackage{xcolor}
\usepackage{array}
\usepackage{tikz}
\usepackage{blkarray, bigstrut}

\newcommand{\un}{\underline}

\renewcommand{\epsilon}{\varepsilon}
\renewcommand{\tilde}{\widetilde}

\newcommand{\F}{{\mathbb F}}

\newcommand{\R}{{\mathbb R}}
\newcommand{\Z}{{\mathbb Z}}
\newcommand{\D}{{\mathbf D}}

\newcommand{\B}{{\mathcal B}}
\newcommand{\T}{{\mathcal T}}

\newcommand{\Sand}{{\mathcal S}}

\newcommand{\what}{\widehat}
\newcommand{\Fl}{\mathcal C}
\newcommand{\Cu}{\what{\mathcal C}}
\newcommand{\col}[1]{x_{#1}}
\newcommand{\hcol}[1]{\widehat x_{#1}}
\newcommand{\indB}[1]{k_{#1}}

\newcommand{\CB}{x_B}
\newcommand{\hCB}{x_{\widehat B}}

\newcommand{\colB}[1]{x_{k_{#1}}}
\newcommand{\hcolB}[1]{\widehat x_{k_{#1}}}
\newcommand{\CS}{x_S}
\newcommand{\w}{\mathfrak w}

\newcommand{\Ima}{\operatorname{im}}

\newcommand{\Cut}{\operatorname{Cocirc}}
\newcommand{\Flow}{\operatorname{Circ}}

\DeclareMathOperator{\rk}{rk}

\DeclareMathOperator{\spn}{span}

\usetikzlibrary{arrows.meta}
\usetikzlibrary{calc}
\usetikzlibrary{patterns}

\title
{A family of matrix-tree multijections}

\author[\initial{A.} McDonough]{\firstname{Alex} \lastname{McDonough}}

\address{Brown University\\ 
Department of Mathematics\\
151 Thayer St.\\
Providence\\
RI 02912 (USA)}

\email{amcd2718@gmail.com}


\begin{document}
\begin{abstract}
  For a natural class of $r \times n$ integer matrices, we construct a non-convex polytope which periodically tiles $\R^n$. From this tiling, we provide a family of geometrically meaningful maps from a generalized sandpile group to a set of generalized spanning trees which give multijective proofs for several higher-dimensional matrix-tree theorems. In particular, these multijections can be applied to graphs, regular matroids, cell complexes with a torsion-free spanning forest, and representable arithmetic matroids with a multiplicity one basis. This generalizes a bijection given by Backman, Baker, and Yuen and extends work by Duval, Klivans, and Martin.

\end{abstract}

\maketitle

\section{Introduction}

Given a connected graph $G$, the \textit{sandpile group} $\Sand(G)$ is a finite abelian group related to a discrete dynamical system. This group, and the related \textit{abelian sandpile model}, have been applied to a wide variety of subjects, such as algebraic geometry, electrical networks, and statistical mechanics~\cite{Lorenzini,Electric,BTW}. In different contexts, the sandpile group is also called the \textit{critical group}, \textit{graph Jacobian}, \textit{graph Picard group}, or \textit{group of components}. 

One striking property of $\Sand(G)$ is that its size is equal to the number of \textit{spanning trees} of $G$. This relationship follows from Kirchhoff's matrix-tree theorem, a classical graph theoretical result with many generalizations (see~\cite{Seth}). While this theorem implies the existence of bijections between $\Sand(G)$ and $\mathcal T(G)$, the standard proof is not bijective. There has been a great deal of interest in providing combinatorially meaningful bijections between these two sets. See, for example,~\cite{Dhar92,BigWink,Holroyd,Bernardi}.

The sandpile group, spanning trees, and the matrix-tree theorem can all be generalized to larger classes of objects such as regular matroids (see \cite{Merino,Gioan2,Gioan}) and cell complexes (see \cite{Simp,Crit,CutsFlows}). For this paper, our primary objects of interest will be a class of integer matrices called \textit{standard representative matrices} (see Definition~\ref{def:standrep}). In the author's dissertation, he shows that any graph, regular matroid, cell complex with a torsion-free spanning forest, or orientable arithmetic matroid with a multiplicity one basis is associated with a standard representative matrix~\cite{mythesis}. 

Let $D$ be a standard representative matrix (see Definition~\ref{def:standrep}). In Section~\ref{sec:Background}, we define the \textit{sandpile group} $\Sand(D)$, the \textit{bases} $\B(D)$, and the \textit{basis multiplicity function} $m$ which maps each $B \in \B(D)$ to a positive integer. In this context, we get the following theorem, which is a reframing of Theorem 8.1 from~\cite{CutsFlows}.

\begin{theorem}[Sandpile matrix-tree theorem on standard representative matrices]\label{thm:OAmtt} 
\[|\Sand(D)| = \sum_{B \in \mathcal B(D)} m(B)^2.\]
\end{theorem}

When $D$ is associated with a regular matroid, $m(B) = 1$ for all $B \in \B(D)$ and thus Theorem~\ref{thm:OAmtt} implies that $|\Sand(D)| = |\B(D)|$ (this is Theorem 4.6.1 from~\cite{Merino}). In 2017 (published in 2019), Backman, Baker, and Yuen define a family of geometric bijections between $\Sand(D)$ and $\B(D)$ for the regular matroid case~\cite{BBY,Yuen}. However, their construction does not easily generalize to the case where not all bases have multiplicity 1.

Our main result is Theorem~\ref{thm:Hellyeah}, which gives the analogue of a bijection for an arbitrary standard representative matrix. In particular, we define a family of geometrically meaningful maps $f: \Sand(D) \to \mathcal B(D)$ such that for any $B \in \mathcal B(D)$, we have $|f^{-1}(B)| = m(B)^2$. We call these maps \textit{sandpile multijections}. 

Our general construction is geometric, as in~\cite{BBY}. We associate each basis with a parallelepiped of volume $m(B)^2$. These parallelepipeds do not intersect and their union produces a non-convex polyhedron that periodically tiles $\R^{|E|}$. Using our shifting vector, we associate $m(B)^2$ points of $\Z^{|E|}$ to each parallelepiped. Furthermore, we show that these points are all distinct in $\Sand(D)$. 

For the sake of brevity, we restrict our attention to standard representative matrices in this paper. For a more complete story which explores the connection between different kinds of sandpile groups and focuses on \textit{orientable arithmetic matroids}, which were recently defined in~\cite{Pagaria}, see the first nine chapters of the author's dissertation~\cite{mythesis}. This paper consists primarily of material from the seventh and eight chapters. The ninth chapter shows how to obtain multijections on a larger class of matrices when the sandpile group is replaced with its \textit{Pontryagin dual}. 

In Section~\ref{sec:notation}, we go over some notational conventions we will use throughout the paper. In Section~\ref{sec:Background}, we give background on lattices and define standard representative matrices. In Section~\ref{sec:graphsand}, we motivate our future results by constructing a standard representative matrix from a graph. In Section~\ref{sec:tiling}, we show how to construct a periodic tiling of $\R^n$ from any standard representative matrix. In Section~\ref{sec:multi}, we use this tiling to construct a family of sandpile multijections. In Section~\ref{sec:Lower}, we demonstrate how to generate lower-dimensional tilings which produce equivalent multijections. In Section~\ref{sec:shifting}, we show how a choice of \textit{shifting vector} corresponds to a choice of \textit{chamber} from a \textit{hyperplane arrangement}. In Section~\ref{sec:cornp}, we associate certain important points with $\{0,1\}^n$ vectors in the same equivalence class of $\Sand(D)$. Finally, in Section~\ref{sec:quest}, we provide some open questions for further study.  

\section{Notational Conventions}\label{sec:notation}
We will write $\Z$ for the integers and $\R$ for the real numbers. We write $[a,b]$ for the set $\{x\in \Z\mid a \le x \le b\}$ and $[b]$ for $[1,b]$. We denote a vector of all zeros by $\un 0$. We use the variable $D$ for an $r \times n$ integer matrix which, starting in Section~\ref{sec:graphsand}, will always be a \textit{standard representative matrix} (see Definition~\ref{def:standrep}). We write $\what D$ and $\D$ for the \textit{dual matrix} and \textit{full matrix} of $D$ respectively (again, see Definition~\ref{def:standrep}). We will always write the determinant of a square matrix $A$ as $\det(A)$ and use $| \cdot |$ for set cardinality or absolute value. We also write $A^T$ for the transpose of a matrix $A$ and $I_k$ for the $k \times k$ identity matrix. We will frequently be working with polyhedra embedded in $\R^k$ (where $k$ is either $n$, $r$, or $n-r$). We use the term \textit{volume} to mean $k$-dimensional Lebesgue measure.

\section{Background and definitions}\label{sec:Background}

\begin{defi}
A \textit{lattice} is a subgroup of a finite-dimensional vector space that is isomorphic to $\Z^k$ for some $k$.
\end{defi}

A subgroup of a lattice is called a \textit{sublattice}. Given any set $S$ of $\Z^k$ vectors, the integer linear combinations of these vectors form a lattice $L$ of dimension at most $k$. We say that $S$ \textit{generates} $L$. If the vectors in $S$ are linearly independent, we say that $S$ is an \textit{integral basis} for $L$. 

\begin{rema} When working with vector spaces, any maximal linearly independent set of generators is a basis. However, a maximal linearly independent set of generators for a lattice is not always an integral basis of this lattice. For example, the set $\{2,3\}$ generates $\Z$, but neither $\{2\}$ nor $\{3\}$ is an integral basis for $\Z$. 
\end{rema}

\begin{prop}{\cite[Theorem~14.5.3]{Godsil}}\label{cokersize}
If $S$ is a set of $k$ vectors in $\Z^k$ that are an integral basis for a lattice $L$, then the group $\Z^k/ L$ has size equal to the magnitude of the determinant of the matrix formed by the vectors of $S$.
\end{prop}

For a lattice $L$, the group $\Z^k / L$ is called the \textit{cokernel} of $L$. For some integers $r \le n$, let $D$ be an $r \times n$ integer matrix:

\begin{defi}\label{def:cutandflow}\hfill
\begin{itemize}
\item The \textit{cocircuit space} of $D$ is the space $\Ima_\R(D^T)$.
\item The \textit{circuit space} of $D$ is the space $\ker_\R(D)$.
\item The \textit{cocircuit lattice} of $D$ is the lattice $\Ima_\Z(D^T)$.
\item The \textit{circuit lattice} of $D$ is the lattice $\ker_\Z(D)$.
\item The \textit{sandpile lattice} of $D$ is the lattice $\Ima_\Z(D^T) \oplus \ker_\Z(D)$.
\end{itemize}
\end{defi}

\begin{rema} In~\cite{mythesis}, the cocircuit space, circuit space, cocircuit lattice, and circuit lattice are called $\Cut(D)$, $\Flow(D)$, $\Cu(D)$, and $\Fl(D)$ respectively. We omit this additional notation in this paper for the sake of readability. \end{rema}

The cocircuit space and cocircuit lattice are generated by the rows of $D$. The circuit space and circuit lattice are generated by the coefficients of integer linear combinations of columns of $D$ that sum to $\un 0$. Note that these generators are all elements of $\Z^n$.

\begin{rema} When $D$ is the boundary matrix of a graph $G$, $\Ima_\Z(D^T)$ and $\ker_\Z(D)$ are called the \textit{cut lattice} of $G$ and the \textit{flow lattice} of $G$ respectively. These lattices were first defined in~\cite{Bacher}. Here, the cokernel of the sandpile lattice is isomorphic to the usual sandpile group of the graph (as we will discuss in Section~\ref{sec:graphsand}). Similarly, when $D$ is the boundary matrix of a cell complex $\Sigma$, $\Ima_\Z(D^T)$ and $\ker_\Z(D)$ are the cut lattice and flow lattice of $\Sigma$ as defined in~\cite{CutsFlows}. Duval, Klivans, and Martin call the sandpile lattice the \textit{cutflow lattice} and its cokernel the \textit{cutflow group}. 
\end{rema}

\begin{lemma}[{\cite[Proposition 5.1]{CutsFlows}}]\label{lem:orthogonality}
For any integer matrix $D$, the spaces $\Ima_\R(D^T)$ and $\ker_\R(D)$ are orthogonal complements.  
\end{lemma}

Note that because $\Ima_\Z(D^T) \subset \Ima_\R(D^T)$ and $\ker_\Z(D) \subset \ker_\R(D)$, this also means that $\Ima_\Z(D^T)$ and $\ker_\Z(D)$ are always orthogonal. We also get the following corollary:

\begin{coro}\label{cor:cutofflow}
Let $D$ be an integer matrix and $\what D$ be a matrix with rows that generate $\ker_\R(D)$. $\ker_\R(D) = \Ima_\R(\what D^T)$ and $\ker_\R(\what D) = \Ima_\R(D^T)$. 
\end{coro}

\begin{proof}
The first equality follows immediately from the fact that $\Ima_\R(\what D^T)$ is generated by the rows of $\what D$ which also generate $\ker_\R(D)$ by definition. 

For the second equality, by Lemma~\ref{lem:orthogonality}, $\ker_\R(\what D)$ is the orthogonal complement of $\Ima_\R(\what D^T)$ which we established is equal to $\ker_\R(D)$. By a second application of Lemma~\ref{lem:orthogonality}, $\ker_\R(D)$ is the orthogonal complement of $\Ima_\R(D^T)$. Since the composition of two orthogonal complements is the identity, we conclude that $\ker_\R(\what D) = \Ima_\R(D^T)$. 
\end{proof}

\begin{defi}\label{def:standrep}
An $r \times n$ integer matrix $D$ is a \textit{standard representative matrix} if it is of the form:
\[D = \begin{pmatrix} I_r & M\end{pmatrix},\]
where $I_r$ is the $r \times r$ identity matrix and $M$ is any $r \times (n-r)$ integer matrix. A standard representative matroid $D$ is associated with two other matrices:
\[\what D = \begin{pmatrix} -M^T & I_{n-r}\end{pmatrix} \hspace{.4 cm}\text{ and }\hspace{.4 cm} \D = \begin{pmatrix} D \\ \what D \end{pmatrix} = \begin{pmatrix}  I_r & M\\  -M^T & I_{n-r}\end{pmatrix}.\]
We call $\what D$ the \textit{dual matrix} of $D$ and $\D$ the \textit{full matrix} of $D$. We will show in Lemma~\ref{lem:flowofcut} that our notation for $\what D$ is consistent with Corollary~\ref{cor:cutofflow}. 
\end{defi}

\begin{rema}
The term \textit{standard representative matrix}, which appears in~\cite[Section 2.2]{Oxley}, is named for the fact that every \textit{representable matroid} can be \textit{represented} by a matrix of this form (after rearranging columns). However, it is worth noting that we can only represent \textit{oriented arithmetic matroids} using a matrix of this form if they have a \textit{basis of multiplicity one} (see~\cite[Corollary 4.3.13]{mythesis}).\footnote{We also need to restrict to oriented arithmetic matroids satisfying the \textit{strong GCD property} or else not all oriented arithmetic matroids are representable (see~\cite[Section 4.2]{mythesis}). For this paper, whenever we mention oriented arithmetic matroids, we will always assume this property.} In~\cite[Chapter 5]{mythesis}, the set of representations for an arbitrary oriented arithmetic matroid are classified. 
\end{rema}
\begin{lemma}{\cite[Corollary 4.6.6]{mythesis}}\label{lem:flowofcut}
If $D$ is a standard representative matrix, then $\ker_\Z(D) = \Ima_\Z(\what D^T)$, $\Ima_\Z(D^T) = \ker_\Z(\what D)$, and $\Ima_\Z(\D^T) = \Ima_\Z(D^T) \oplus \ker_\Z(D)$. 
\end{lemma}

\begin{defi}\label{def:sandpile}
The \textit{sandpile group} of a standard representative matrix $D$, denoted $\Sand(D)$, is the finite abelian group
\[\Sand(D) = \Z^n / (\Ima_\Z(D^T) \oplus \ker_\Z(D)) = \Z^n / (\Ima_\Z(\D^T)).\]
\end{defi}

Notice that by Lemma~\ref{lem:flowofcut}, the sandpile group of $D$ is the cokernel of the sandpile lattice of $D$. 

\begin{defi}\label{def:bandm}
The set of \textit{bases} of $D$, written $\mathcal B(D)$, is the set of $r$-tuples of columns of $D$ such that the determinant of $D$ restricted to these columns is nonzero. For $B \in \mathcal B(D)$, let $m(B)$ be the absolute value of this determinant. This $m(B)$ is called the \textit{multiplicity} of $B$. 
\end{defi}

\begin{rema} These definitions come from the theory of \textit{arithmetic matroids}. In \cite{CutsFlows}, the authors work with \textit{cell complexes} instead of standard representative matroids (although they note in Remark 4.2 that their ideas can be translated to an integer matrix context). Our bases correspond to what they call \textit{cellular spanning forests}, basis multiplicity correspond to the size of the \textit{torsion subgroup} of a certain \textit{relative homology}, and the sandpile group corresponds to what they call the \textit{cutflow group}. See~{\cite[Section 6.6]{mythesis}} for more discussion on the sandpile group of a cell complex and how this relates to the sandpile group of a standard representative matrix. 
\end{rema}

Recall that the sandpile matrix-tree theorem for standard representative matrices (Theorem~\ref{thm:OAmtt}) says that:
\[|\Sand(D)| = \sum_{B \in \mathcal B(D)} m(B)^2.\]

In the following example, we give a demonstration of this theorem. 
\begin{exam}\label{ex:fave0}
Suppose that for $r=2$ and $n=3$, we have the following standard representative matrix:
\[ D = \begin{pmatrix}1 & 0 & 3\\ 0 & 1 & 2\end{pmatrix}.\]
Because all of the maximal minors are nonzero, $\B(D) = \{\{1,2\},\{1,3\},\{2,3\}\}$. Furthermore, $m(\{1,2\}) = 1$, $m(\{1,3\}) = 2$, and $m(\{1,2\}) = 3$. Theorem~\ref{thm:OAmtt} says that $|\Sand(D)| = 1^2 + 2^2 + 3^2 = 14$. Recall that by definition of $\Sand(D)$, this is the number of elements in $\Z^3 / \Ima_\Z(\D^T)$ where
\[ \D = \begin{pmatrix}1 & 0 & 3\\ 0 & 1 & 2\\-3 & -2 & 1\end{pmatrix}.\]
\end{exam}
\begin{defi}
An \textit{$\mathfrak m$-multijection} between sets $S$ and $T$ is a map $f:S \to T$ such that for all $t \in T$, $|f^{-1}(t)| = \mathfrak m(t)$. \end{defi}

An $\mathfrak m$-multijection can also be thought of as a bijection between $S$ and a multiset consisting of $\mathfrak m(t)$ copies of each $t \in T$. In this paper, we give an explicit procedure for constructing $\mathfrak m$-multijections between $\Sand(D)$ and $\B(D)$ for $\mathfrak m(B) = m(B)^2$. To do this, we use a geometric construction, which also produces a periodic tiling of $\R^n$. 

\section{Graphs and Standard Representative Matrices}\label{sec:graphsand}

In this section, we show how to obtain a standard representative matrix from a graph $G$ and one of its spanning trees. The results for this section will not be necessary for understanding future sections, but they are intended to provide some context for the relevance of standard representative matrices. For a more thorough analysis of the connection between standard representative matrices and other objects, see~\cite[Chapters 3-6]{mythesis}. 

Throughout this section, we will fix a finite connected undirected graph $G$ with edges $E(G)$ and spanning trees $\T(G)$ (i.e. maximal collections of edges containing no cycles). Let $n = |E(G)|$ and $r =|T|$ for every $T \in \T(G)$ (it is a classical result that all spanning trees of a graph contain the same number of edges). Furthermore, we will write the edges of $G$ as $\{e_1,\dots,e_n\}$ such that $\{e_1,\dots,e_r\}$ forms a spanning tree which we call $T$. 
\begin{defi}\hfill
\begin{itemize}
    \item A \textit{circuit} of $G$ is a minimal (by inclusion) subset of $E(G)$ not contained in any spanning tree. 
    \item A \textit{cocircuit} of $G$ is a minimal (by inclusion) subset of $E(G)$ intersecting every spanning tree. 
\end{itemize}
These definitions come from matroid theory. In the graphic context, circuits are also called \textit{cycles} and cocircuits are also called \textit{bonds} or \textit{minimal cuts}. 
\end{defi}

\begin{lemma}[{\cite[Corollary 1.2.6,~Exercise 2.1.10]{Oxley}}]\label{lem:funcirc}\hfill
\begin{itemize}
\item For any $e \in E(G) \setminus T$, the set of edges $T \cup \{e\}$ contains a unique circuit. 
\item For any $e \in T$, the set of edges $(E(G) \setminus T) \cup \{e\}$ contains a unique cocircuit. 
\end{itemize}
\end{lemma}

\begin{defi}\hfill
\begin{itemize}
\item For any $e \in E(G) \setminus T$, the circuit contained in $T \cup \{e\}$ is called it \textit{fundamental circuit of $e$} and is denoted $C^e$.  
\item For any $e \in T$, the cocircuit contained in $(E(G) \setminus T) \cup \{e\}$ is called the \textit{fundamental cocircuit of $e$} and is denoted $\what C^e$.
\end{itemize}
\end{defi}

Choose an arbitrary orientation for the edges of $G$. Note that the orientation is for bookkeeping purposes and one should not think of $G$ as a directed graph. Each circuit on a graph corresponds to a cyclic set of edges (ignoring orientation). For $e_i \in E(G) \setminus T$ and $e_j \in T\cap C^e$, we say that \textit{$e_i$ matches the orientation of $C^{e_j}$} if the edges of $C^{e_j}$ can be cyclically oriented in a way that matches the orientation of both $e_i$ and $e_j$. We define an $r \times n$ matrix $D$ in the following way.  
\[\text{For $j\le r$, }D_{ij} = \begin{cases}
1 & \text{ if $i=j$,}\\
0 & \text{ if $i\not= j$.}
\end{cases}\]\[\text{For $j> r$, }D_{ij} = \begin{cases}
1 & \text{ if $e_i \in C^{e_j}$ and $e_i$ matches the orientation of $C^{e_j}$,}\\
-1 & \text{ if $e_i \in C^{e_j}$ and $e_i$ does not match the orientation of $C^{e_j}$,}\\
0 & \text{ otherwise.}\\
\end{cases}\]

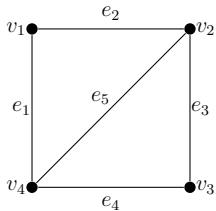
\begin{figure}
\begin{center}
\begin{tikzpicture}[scale = 0.7]
    
    \tikzstyle{every node} = [circle,fill,inner sep=1pt,minimum size = 1.5mm]
    \node(a) at (0,0) {};
    \node(b) at (0,3){};
    \node(c) at (3,3) {};
    \node(d) at (3,0){};

    \tikzstyle{every node} = [draw = none,fill = none,scale = .7]
    \draw (a) -- (b);
    \draw (b) -- (c);
    \draw (c) -- (d);
    \draw (a) -- (c);
    \draw (a) -- (d);
    
    \node (o) at (-.3,3){{\Large $v_1$}};
    \node (o) at (3.3,3){{\Large $v_2$}};
    \node (o) at (3.3,0){{\Large $v_3$}};
    \node (o) at (-.3,0){{\Large $v_4$}};
    
    \node (o) at (-.2,3/2){\Large $e_1$};
    \node (o) at (3/2,3.3){\Large $e_2$};
    \node (o) at (3.2,3/2){\Large $e_3$};
    \node (o) at (3/2,-.3){\Large $e_4$};
    \node (o) at (1.3,1.7){\Large $e_5$};

\end{tikzpicture}
\caption{A graph with 4 vertices and 5 edges.}\label{fig:k4minus1}
\end{center}
\end{figure}

\begin{exam}
Let $G$ be the graph in Figure~\ref{fig:k4minus1}. Choose $T = \{e_1,e_2,e_3\}$ and orient each edge from smaller to larger numbered vertex. This gives the following matrix:
\[D = 
\begin{blockarray}{ccccc}
 e_1 & e_2 & e_3 & e_4 & e_5 \\
\begin{block}{(ccccc)}
  1 & 0 & 0 & -1 & -1 \\
  0 & 1 & 0 & 1 & 1 \\
  0 & 0 & 1 & 1 & 0 \\
\end{block}
\end{blockarray}.
 \]
\end{exam}

It follows immediately from construction that the matrix $D$ is always a standard representative matrix. Notice that the construction of $D$ does not require information about the vertices of $G$. This property means that the construction is \textit{matroidal}. From Definition~\ref{def:sandpile}, it is logical to define the sandpile group of $G$ as:
\[\Z^{E(G)}/(\Ima_\Z(D^T) \oplus \ker_\Z(D)),\] 
a subgroup of the free abelian group on the edges of $G$. In Proposition~\cite[4.1.8]{mythesis}, we show that this definition does not depend on the choice of spanning tree $T$.

The usual definition of sandpile group of a graph is a subgroup of the free abelian group on the \textit{vertices} of $G$. We will not define this group here (see e.g. \cite{Klivans}), but we will call it the \textit{vertex sandpile group of G}. The following proposition follows from results in~\cite{Bacher,Biggs} as well as from~\cite[Proposition~3.2.11 and Proposition~4.1.18]{mythesis}.

\begin{prop}
The boundary map between edges and vertices of $G$ (with respect to the orientation we used to define $D$) induces an isomorphism between $\Z^{E(G)}/(\Ima_\Z(D^T) \oplus \ker_\Z(D))$ and the vertex sandpile group of $G$. 
\end{prop}

We can also define an integral basis for $\ker_\Z(D)$ in terms of the fundamental cocircuits of $e$ for $e \in T$. 

Choose the same orientation on $G$ that we used for defining $D$. Each cocircuit on $G$ corresponds to a minimal set of edges which partition the vertices of $G$ into subsets $V_1$ and $V_2$. For $e_i \in T$ and $e_j \in (E(G) \setminus T)\cap \what C^{e_i}$, we say that \textit{$e_i$ matches the orientation of $\what C^{e_j}$} if $e_i$ and $e_j$ are both oriented from $V_1$ to $V_2$ or both oriented from $V_2$ to $V_1$. We define an $(n-r) \times n$ matrix $\what D$ in the following way.  
\[\text{For $j> r$, }\what D_{ij} = \begin{cases}
1 & \text{ if $i=j-r$,}\\
0 & \text{ if $i\not= j-r$.}
\end{cases}\]\[\text{For $j\le r$, }\what D_{ij} = \begin{cases}
1 & \text{ if $e_{i+r} \in \what C^{e_j}$ and $e_{i+r}$ matches the orientation of $\what C^{e_j}$,}\\
-1 & \text{ if $e_{i+r} \in \what C^{e_j}$ and $e_{i+r}$ does not match the orientation of $\what C^{e_j}$,}\\
0 & \text{ otherwise.}\\
\end{cases}\]

We show in~\cite[Lemma~4.5.12]{mythesis} that $\what D$ is the dual matrix of $D$, so our notation is consistent with the notation given in Definition~\ref{def:standrep}. 

\begin{rema}
The construction of $D$ given above can be applied to any~\textit{regular matroid}, and a version of this construction was used in~\cite{BBY}. We can also generalize this definition to any~\textit{cell complex with a torsion-free spanning forest} or~\textit{representable arithmetic matroid with at least one multiplicity one basis}. See~\cite[Chapter 4-6]{mythesis} for more discussion of this generalization. 
\end{rema}

\section{A Tiling of $\R^n$}\label{sec:tiling}

For the remainder of this paper, we will always let $D = \begin{pmatrix} I_r &M \end{pmatrix}$ be an $r \times n$ standard representative matrix. Furthermore, we let $\widehat D = \begin{pmatrix} -M^T & I_{n-r} \end{pmatrix}$ be the dual matrix of $D$ and 
\[\D = \begin{pmatrix} D \\ \what D \end{pmatrix} = \begin{pmatrix}  I_r & M\\  -M^T & I_{n-r}\end{pmatrix}\]
be the full matrix of $D$. Recall from Definition~\ref{def:bandm} that $\B(D)$ is the set of $r$ element subsets of the columns of $D$ with nonzero determinant and for $B \in \B(D)$, $m(B)$ is the magnitude of the corresponding determinant. In this section, we will associate each $B \in B$ with a lattice parallelepiped and then show that the non-convex polytope formed by their union periodically tiles $\R^n$. In the next section, we will show how to use this tiling to construct a family of multijections. 

We think of $B \in \B(D)$ as a set $\{\indB 1,\dots, \indB r\}$ of column indices. These simultaneously describe a set of columns of $D$, $\widehat D$ or $\D$. Because we are working in $\R^n$, it will be useful to allow for a version of the sandpile group whose representatives are real vectors. 

\begin{defi}
The \textit{continuous sandpile group} of $D$ is the group:
\[\tilde \Sand(D) = \R^n / (\Ima_\Z(D^T) \oplus \ker_\Z(D)) = \R^n/\Ima_\Z(\D^T)\]
\end{defi}

We will also make heavy use of the following lemma, which follows immediately from the definition of sandpile groups and continuous sandpile groups of standard representative matrices. 
\begin{lemma}\label{lem:whenequiv}
Let $D$ be an $r \times n$ standard representative matrix. Two vectors $z,z' \in \Z^{n}$ (resp. $\R^{n}$) are equivalent as elements of $\Sand(D)$ (resp. $\tilde \Sand(D)$) if and only if $z - z' \in \Ima_\Z(\D^T)$.
\end{lemma}

We introduce some definitions and notation that can be found in~\cite{Beck}.
\begin{defi}\hfill\label{def:parpip}
\begin{itemize}
    \item The \textit{fundamental parallelepiped} of a square matrix $A$ with column vectors $\{\col 1,\dots,\col k\}$ is the set of points:
\[ \left\{\sum_{i=1}^k a_i \col i \mid 0 \le a_i \le 1\right\}.\]
    \item The \textit{half-open fundamental parallelepiped} of a square matrix $A$ with column vectors $\{\col 1,\dots,\col k\}$ is the set of points:
\[ \left\{\sum_{i=1}^k a_i \col i \mid 0 \le a_i < 1\right\}.\]
\end{itemize}
We use the notation $\Pi_{\bullet}(A)$ to indicate the fundamental parallelepiped of $A$ and  
$\Pi_{\circ}(A)$ to indicate the half-open fundamental parallelepiped of $A$. 
\end{defi}

It is a classical result that the volume of $\Pi_\bullet(A)$ or $\Pi_\circ(A)$ is the magnitude of $\det(A)$. 

\begin{defi} \label{def:p1p2p}
For any basis $B\in \B(D)$:
\begin{itemize}
    \item $P_1(B)$ is the fundamental parallelepiped of $D$ restricted to columns in $B$. 
    \item $P_2(B)$ is the fundamental parallelepiped of $\widehat D$ restricted to columns \textit{not} in $B$.
    \item $P(B)$ is the direct product of $P_1(B)$ and $P_2(B)$.
\end{itemize}
\end{defi}

Note that $P_1(B)$ is $r$-dimensional, $P_2(B)$ is $(n-r)$-dimensional, and $P(B)$ is $n$-dimensional. 

\begin{lemma}[{\cite[Lemma 7.1.5]{mythesis}}]\label{lem:pipedvol}For any basis $B\in \mathcal B(D)$, $P_1(B)$ and $P_2(B)$ each have volume $m(B)$ while $P(B)$ has volume $m(B)^2$.
\end{lemma}

We can also describe $P(B)$ in the following way.
For each column of $\D$, if this column corresponds to an index of $B$, replace the last $(n-r)$ entries with 0's. If this column does not correspond to an index of $B$, replace the first $r$ entries with 0's. The fundamental parallelepiped of this matrix is $P(B)$. See Example~\ref{ex:favorite}. 

\begin{exam}\label{ex:favorite}
Consider the matrix 

\[D = \begin{pmatrix}
1 & 0 & 3 \\
0 & 1 & 2\\
\end{pmatrix}
\text{ which is associated with the matrix }
\D = \begin{pmatrix}
1 & 0 & 3 \\
0 & 1 & 2\\
-3&-2 & 1\\
\end{pmatrix}.\]

\noindent As we saw in Example~\ref{ex:fave0}, there are 3 bases of $\mathcal B(D)$, one for every pair of columns. The associated parallelepipeds are given below:

\[
P_1(\{1,2\}) = \Pi_\bullet\begin{pmatrix}
1 & 0 \\
0 & 1 \\
\end{pmatrix}\hspace{.5cm}
P_2(\{1,2\}) = \Pi_\bullet\begin{pmatrix}
1 \\
\end{pmatrix}\hspace{.5cm}
P(\{1,2\}) = \Pi_\bullet\begin{pmatrix}
1 & 0 & 0\\
0 & 1 & 0\\
0 & 0 & 1\\
\end{pmatrix}\]

\[
P_1(\{1,3\}) = \Pi_\bullet\begin{pmatrix}
1 & 3 \\
0 & 2 \\
\end{pmatrix}\hspace{.5cm}
P_2(\{1,3\}) = \Pi_\bullet\begin{pmatrix}
-2 \\
\end{pmatrix}\hspace{.5cm}
P(\{1,3\}) = \Pi_\bullet\begin{pmatrix}
1 & 0 & 3\\
0 & 0 & 2\\
0 & -2 & 0\\
\end{pmatrix}\]

\[
P_1(\{2,3\}) = \Pi_\bullet\begin{pmatrix}
0 & 3 \\
1 & 2 \\
\end{pmatrix}\hspace{.5cm}
P_2(\{2,3\}) = \Pi_\bullet\begin{pmatrix}
-3 \\
\end{pmatrix}\hspace{.5cm}
P(\{2,3\}) = \Pi_\bullet\begin{pmatrix}
0 & 0 & 3\\
0 & 1 & 2\\
-3& 0 & 0\\
\end{pmatrix}.\]

\begin{figure}
\begin{center}

\begin{tikzpicture}
    \tikzstyle{every node} = [circle,fill,inner sep=1pt,minimum size = 1.5mm];
\node(aaa) at (0,0) {};
\node(aba) at (0,1){};
\node(bba) at (1,1){};
\node(baa) at (1,0){}; 
\node(aab) at (.3,.6){};
\node(abb) at (.3,1.6){};
\node(bbb) at (1.3,1.6){};
\node(bab) at (1.3,.6){};

\node(c) at (3.9,1.8){};
\node(d) at (3.6,1.2){};
\node(e) at (4.6,1.2){};

\node(f) at (0,-2){};
\node(g) at (3.6,-.8){};
\node(h) at (4.6,-.8){};
\node(i) at (1,-2){};

\node(j) at (0,-3){};
\node(k) at (3.6,-1.8){};
\node(l) at (3.9,-1.2){};
\node(m) at (0.3,-2.4){};

\draw (aaa) -- (aba) -- (bba) -- (baa) -- (aaa) ;
\draw (abb) -- (bbb) -- (bab);
\draw[dashed] (bab)--(aab);
\draw[dashed] (aaa)--(aab);
\draw[dashed] (aab)--(abb);
\draw(aba) -- (abb);
\draw(bba) -- (bbb);
\draw(baa) -- (bab);

\draw[dashed] (aab)--(1.3,0.6+1/3);
\draw (1.3,0.6+1/3)-- (c) -- (d) -- (1.2,0.4);
\draw[dashed] (1.2,0.4)--(aaa);
\draw (baa)--(e)--(d);
\draw (aaa)--(f) -- (i) -- (baa);
\draw (i) -- (h) -- (e);
\draw[dashed](f) -- (g) -- (d);
\draw[dashed](g) --(h);

\draw (f)--(j) -- (k) -- (l);
\draw[dashed] (j) -- (m) -- (l);
\draw[dashed] (m) -- (aab);
\draw (k) -- (3.6,-1.133);
\draw [dashed] (3.6,-1.133) -- (g);
\draw(c) -- (3.9,1.2);
\draw[dashed] (3.9,1.2) -- (l);
\end{tikzpicture}

\caption{Here is a plot of the three parallelepipeds from Example~\ref{ex:favorite} in $3$-dimensional space. The cube is $P(\{1,2\})$, the smaller of the two remaining parallelepipeds is $P(\{1,3\})$, and the larger is $P(\{2,3\})$. We will see in Corollary~\ref{cor:tiling} that the union of these parallelepipeds periodically tiles the plane.}
\label{favefig}
\end{center}
\end{figure}
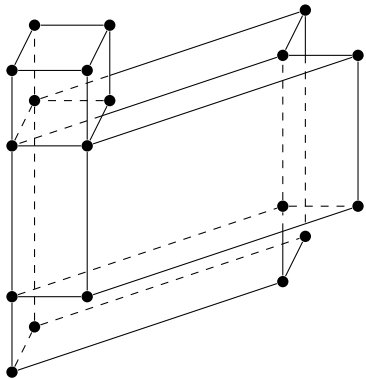

\end{exam}

See Figure~\ref{favefig} for a plot of these three parallelepipeds. Notice that they only intersect at their boundaries. We show that this is true in general. 

\begin{prop}\label{prop:nointersect}
The parallelepipeds $P(B)$ for each basis $B\in \B(D)$ do not intersect except at their boundaries.
\end{prop}
\begin{proof}
Let $\{\col 1,\dots, \col n\}$ be the columns of $D$ and $\{\hcol 1,\dots, \hcol n\}$ be the columns of $\widehat D$. Let $B_1$ and $B_2$ be two distinct bases in $\B(D)$. 

$P(B_1)$ and $P(B_2)$ have intersecting interiors if and only if $P_1(B_1)$ and $P_1(B_2)$ have intersecting interiors and $P_2(B_1)$ and $P_2(B_2)$ have intersecting interiors. Assume that $P_1(B_1)$ and $P_1(B_2)$ have intersecting interiors. Then, for some coefficients $a_1,\dots,a_{n},\\
b_1,\dots,b_{n}$, we have the following equality
\[\sum_{i=1}^{n} a_i \col i = \sum_{i=1}^{n} b_i \col i\]

\noindent where $0 < a_i < 1$ for $i \in B_1$, $a_i = 0$ for $i \not\in B_1$, $0 < b_i < 1$ for $i \in B_2$, and $b_i =0 $ for $i \not\in B_2$. If we subtract the second sum from the first and define $d_i = a_i - b_i$, we get the equation
\[\sum_{i=1}^{n} d_i \col i=\un 0.\]

The above equation implies that
\[(d_1,d_2,\dots,d_n)\in \ker_\mathbb R (D).\]

Similarly, if $P_2(B_1)$ and $P_2(B_2)$ have intersecting interiors then for some coefficients $\widehat a_1,\dots,\widehat a_{n},
\widehat b_1,\dots,\widehat b_{n}$, we have the following equality:
\[\sum_{i=1}^{n} \widehat a_i\hcol i = \sum_{i=1}^{n} \widehat b_i\hcol i\]

\noindent where $0 < \widehat a_i <1$ for $i \not\in B_1$, $a_i = 0$ for $i \in B_1$, $0 < b_i < 1$ for $i \not\in B_2$, and $b_i =0 $ for $i \in B_2$. For each $i \in [n]$, let $\widehat d_i = \widehat a_i - \widehat b_i$. Then,
\[\sum_{i=1}^{n} \widehat d_i\hcol i=\un 0.\]

It follows that:
\[(\widehat d_1,\widehat d_2,\dots,\widehat d_n)\in \ker_\mathbb R (\widehat D) = \Ima_\R(D^T)\]

\noindent where the last equality follows from Lemmas~\ref{lem:flowofcut}.

Lemma~\ref{lem:orthogonality} says that $\Ima_\R(D^T)$ and $\ker_\R(D)$ are orthogonal. This means,
\[0 =  (d_1,d_2,\dots,d_n) \cdot(\widehat d_1,\widehat d_2,\dots,\widehat d_n) = \sum_{i=1}^{n} d_i\widehat d_i.\]

For each $i$, there are 4 possibilities:\\

\textbf{Case 1) $i \in B_1 \cap B_2$}:

$\widehat a_i = \widehat b_i = 0$, so $\widehat d_i = 0$ and $d_i \cdot \widehat d_i = 0$. 

\textbf{Case 2) $i \in B_1 \setminus B_2$}:

$b_i = 0$, so $d_i = a_i$. This means that $0 < d_i < 1$. Furthermore, $\widehat a_i = 0$, so $\widehat d_i = -\widehat b_i$ and $-1 < \widehat d_i  < 0$. It follows that $d_i \cdot \widehat d_i < 0$.

\textbf{Case 3) $i \in B_2 \setminus B_1$}:

$a_i = 0$, so $d_i = -b_i$. This means that $-1 < d_i < 0$. Furthermore, $\widehat b_i = 0$, so $\widehat d_i = \widehat a_i$ and $0 < \widehat d_i < 1$. It follows that $d_i \cdot \widehat d_i < 0$.

\textbf{Case 4) $i \not\in B_1 \cup B_2$}:

$ a_i = b_i = 0$, so $d_i = 0$ and $d_i \cdot \widehat d_i = 0$.\\

$B_1$ and $B_2$ are the same size and distinct, so cases 2 and 3 must each occur at least once. This means that
\[\sum_{i=1}^{n} d_i\widehat d_i < 0.\]
This is a contradiction.
\end{proof}

\begin{defi}
$T(D)$, the \textit{tile associated with $D$}, is 
\[T(D) = \bigcup_{B \in \B(D)} P(B).\]
\end{defi}

\noindent Corollary~\ref{cor:tiling} will justify why we call this non-convex polyhedron a \textit{tile}. 

The following corollary follows directly from Lemma~\ref{lem:pipedvol} which gives the size of each $P(B)$ and Proposition~\ref{prop:nointersect} which says that they don't intersect.

\begin{coro}\label{cor:voltile}
The volume of $T(D)$ is equal to
\[\sum_{B \in \B(D)} m(B)^2.\]
\end{coro}
\noindent Note that this sum is also equal to $|\Sand(D)|$ by Theorem~\ref{thm:OAmtt}. 

When considering all of $T(D)$, we can strengthen Proposition~\ref{prop:nointersect} to the following:

\begin{prop}\label{prop:lastDomino}
Two distinct points of $T(D)$ can only be equivalent as elements of $\tilde \Sand(D)$ if they are both on the boundary of $T(D)$. 
\end{prop}

\begin{proof}
First, we show that two points of $T(D)$ can only be equivalent as elements of $\tilde \Sand(D)$ if they are each on the boundary of some $P(B)$. 

For some $B_1,B_2 \in \B(D)$, let $p_1$ and $p_2$ be interior points of $P(B_1)$ and $P(B_2)$ respectively. Using the notation and reasoning from Proposition~\ref{prop:nointersect}, we can write $p_1 - p_2$ as the vector whose first $r$ entries are given by 
\[\sum_{i=1}^{n} d_i\col i,\]

\noindent and whose last $n-r$ entries are given by 
\[\sum_{i=1}^{n} \widehat d_i\hcol i.\]

By Lemma~\ref{lem:whenequiv}, $p_1$ and $p_2$ are equivalent as elements of $\tilde S(D)$ if and only if:
\[ p_1 - p_2 = \D^T(z_1,\dots, z_{n})^T\]

\noindent for some $(z_1,\dots z_{n}) \in \Z^{n}$. 

Let $s_i$ be the restriction of the $i^{th}$ row of $\D$ to the first $r$ entries and $\widehat s_i$ be the restriction of the $i^{th}$ row of $\D$ to the last $n-r$ entries. Then, the first $r$ entries of
\[\D^T(z_1,\dots, z_{n})^T\]
\noindent are given by
\[ \sum_{i=1}^{n} z_is_i,\]
\noindent and the last $n-r$ entries are given by
\[\sum_{i=1}^{n} z_i\widehat s_i.\]
From the structure of $\D$, $s_i$ and $\col i$ as well as $\widehat s_i$ and $\hcol i$ are closely related. In particular, for $i \in [1,r]$, we have $s_i = \col i$ and $\widehat s_i = -\hcol i$. For $i \in [r+1,n]$, we have $s_i = -\col i$ and $\widehat s_i = \hcol i$. 

This means that the first $r$ entries of
\[\D^T(z_1,\dots z_r,-z_{r+1},\dots -z_{n})^T\]
\noindent are given by
\[\sum_{i=1}^{n} z_i \col i,\]
\noindent and the last $(n-r)$ entries are given by
\[\sum_{i=1}^{n} -z_i\hcol i.\]
Hence the points $p_1$ and $p_2$ are equivalent as elements of $\tilde \Sand(D)$ if and only if we have:
\[\sum_{i = 1}^{n}(d_i-z_i) \col i = \un 0 \hspace{ .4 cm} \text{ and } \hspace{ .4 cm} \sum_{i = 1}^{n}(\widehat d_i+z_i) \hcol i =\un 0.\]
By the same logic that we used for Proposition~\ref{prop:nointersect}, the coefficients of the first sum form an element of $\ker_\R(D)$ while the coefficients of the second form an element of $\Ima_\R(D^T)$. Lemma~\ref{lem:orthogonality} again tells us that their dot product is 0. In other words:
\[\sum_{i = 1}^{n}(d_i-z_i) \cdot (\widehat d_i + z_i) =0.\]
For each $i$, there are 4 possibilities:\\

\textbf{Case 1) $i \in B_1 \cap B_2$}:

$\widehat a_i = \widehat b_i = 0$, so $\widehat d_i = 0$. $0 < a_i,b_i <1$ so $-1 < d_i < 1$. If $z_i = 0$, then $(d_i-z_i) \cdot (\widehat d_i + z_i) = 0$. Otherwise, the two factors have a different sign and the product is negative.  

\textbf{Case 2) $i \in B_1 \setminus B_2$}:

$b_i = 0$, so $d_i = a_i$. This means that $0 < d_i <1$. $\widehat a_i = 0$, so $\widehat d_i = -\widehat b_i$. It follows that $-1 < \widehat d_i <0$. If $z_i > 0 $, then $d_i-z_i < 0$ and $\widehat d_i+ z_i > 0$. If $z_i \le 0$, then $d_i-z_i < 0$ and $\widehat d_i+ z_i > 0$. In either case, $(d_i-z_i) \cdot (\widehat d_i + z_i) < 0$.

\textbf{Case 3) $i \in B_2 \setminus B_1$}:

$a_i = 0$, so $d_i = -b_i$. This means that $-1 < d_i < 0$. $\widehat b_i = 0$, so $\widehat d_i = \widehat a_i$. It follows that $0 < \widehat d_i < 1$. If $z_i \ge 0 $, then $d_i-z_i < 0$ and $\widehat d_i+ z_i > 0$. If $z_i < 0$, then $d_i-z_i < 0$ and $\widehat d_i+ z_i > 0$. In either case, $(d_i-z_i) \cdot (\widehat d_i + z_i) < 0.$

\textbf{Case 4) $i \not\in B_1 \cup B_2$}:

$ a_i = b_i = 0$ so $d_i = 0$. $0 < \widehat a_i,\widehat b_i < 1$ so $-1 < \widehat d_i < 1$. If $z_i = 0$, then $(d_i-z_i) \cdot (\widehat d_i + z_i) = 0$. Otherwise, the two factors have a different sign and the product is negative. \\

In all four cases the product is negative, unless we are always in case 1 or case 4 and $z_i = 0$ for all $i$. However, if $z_i = 0$ for all $i$, then $p_1 = p_2$. Thus, our claim holds by contradiction.

We showed that two distinct points $p_1$ and $p_2$ of $T(D)$ that are equivalent as elements of $\tilde \Sand(D)$ must each lie on the boundary of some $P(B)$. We now show by contradiction that they are on both on the boundary of $T(D)$.

Assume that $p_1$ is an interior point of $T(D)$. Since $T(D)$ is the union of non-degenerate parallelepipeds, there is some vector $w\in \R^n$ such that for all sufficiently small $\epsilon>0$, $p_2 + \epsilon w$ is in $T(D)$ but not on the boundary of any $P(B)$. If we make $\epsilon$ small enough, $p_1 + \epsilon w$ must be in $T(D)$ as well, since $p_1$ is an interior point of $T(D)$ by assumption. Moreover, $p_1 + \epsilon w$ and $p_2+ \epsilon w$ are equivalent as elements of $\tilde \Sand(D)$. We get a contradiction because both points are in $T(D)$, but $p_2+\epsilon w$ is not on the boundary of any $P(B)$. This means that $p_1$ and $p_2$ must both be on the boundary of $T(D)$. 
\end{proof}

The next corollary shows that copies of $T(D)$ can be used to periodically tile $\R^{n}$.

\begin{coro}\label{cor:tiling}
The set of translates $T(D) + \D^T(z_1,\dots, z_{n})^T$ for all $(z_1,\dots, z_{n}) \in \Z^{n}$ cover all of $\R^{n}$ and only intersect at their boundaries. 
\end{coro}

\begin{proof}
Consider any point $p\in \R^{n}$. By Lemma~\ref{lem:whenequiv}, the points which are equivalent to $p$ as elements of $\tilde S(D)$ are those of the form $p + \D^T(z_1,\dots, z_{n})$ for $(z_1,\dots, z_{n}) \in \Z^{n}$. Since these are exactly the translates of $T(D)$, the condition that the translates do not intersect except at their boundaries follows directly from Proposition~\ref{prop:lastDomino}.

We also have to show that the translates cover all of $\R^{n}$ given that they do not overlap except at their boundaries. We first note that $\Pi_\circ(\D^T)$ must tile $\R^{n}$ under the same translation because for every $p \in \R^{n}$, there is a unique solution to $(\D^T)p' = p$ (in particular $p' = (\D^T)^{-1} p$). We can map each point of $T(D)$ to a point in $\Pi_\circ(\D^T$) by translating it by an integer combination of columns of $\D^T$. Let $t$ be this piecewise translation from $T(D) \to \Pi_\circ(\D^T)$. Each translation preserves the volume of the region we transform and the only overlap is from the boundary of $T(D)$, which is a 0 volume set. It follows that the volume of the image of $t$ is equal to the volume of $T(D)$. Since $\Pi_\circ(\D^T$) has the same volume as $T(D)$, the set of points that are not in the image of $t$ must have volume $0$. 

Let $p$ be a point of $\Pi_\circ(\D^T$) that is not in the image of $t$. The preimage of $p$ is the collection of points in the same equivalence class with respect to $\tilde \Sand(D)$. By assumption, none of these points are in $T(D)$. Since $T(D)$ is closed, this means that none of these points are limit points of $T(D)$ either, so there is a neighborhood of $p$ that is also not in the image of $t$. However, this neighborhood must have positive volume, which is a contradiction. 
\end{proof}

\begin{exam}\label{ex:2dim}
The simplest case is when $r=1$ and $n=2$. Here, $\D$ is of the form:
\[ \D = \begin{pmatrix}
1 & k \\
-k & 1 \\
\end{pmatrix}\]

\noindent for some integer $k$. When $k = 3$, we get the pattern in Figure~\ref{fig:tiling}.

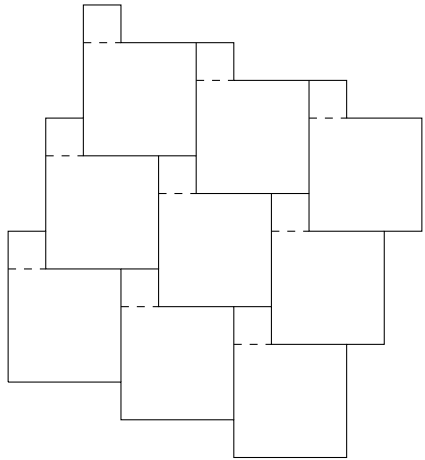
\begin{figure}
\begin{center}
\begin{tikzpicture}[scale = 0.5]
    \tikzstyle{every node} = [shape = none,fill=none,inner sep=1pt,minimum size = .1mm];

\draw (0,2) -- (0,6) --(1,6) -- (1,9)--(2,9) -- (2,12) -- (3,12) -- (3,11) -- (6,11) -- (6,10) -- (9,10) -- (9,9) -- (11,9) -- (11,6) -- (10,6) -- (10,3) -- (9,3) -- (9,0) -- (6,0) -- (6,1) -- (3,1) -- (3,2) -- (0,2);

\draw (1,6) -- (1,5) -- (3,5) -- (3,2);
\draw (3,5) -- (4,5) -- (4,4) -- (6,4) -- (6,1);
\draw (6,4) -- (7,4) -- (7,3) -- (9,3);
\draw (2,9) -- (2,8) -- (4,8) -- (4,5);
\draw (4,8) -- (5,8) -- (5,7) -- (7,7) -- (7,4);
\draw (7,7) -- (8,7) -- (8,6) -- (10,6);
\draw (5,11) -- (5,8);
\draw (8,10) -- (8,7);

\draw[dashed] (0,5) -- (1,5);
\draw[dashed] (1,8) -- (2,8);
\draw[dashed] (2,11) -- (3,11);
\draw[dashed] (3,4) -- (4,4);
\draw[dashed] (4,7) -- (5,7);
\draw[dashed] (5,10) -- (6,10);
\draw[dashed] (6,3) -- (7,3);
\draw[dashed] (7,6) -- (8,6);
\draw[dashed] (8,9) -- (9,9);
\end{tikzpicture}
\caption{Above are 9 copies of $T(D)$ for $D = \protect\begin{pmatrix} 1 & 3 \protect\end{pmatrix}$. 
The dashed lines indicate the boundary between the parallelepipeds that make up $T(D)$ while the solid lines indicate the boundary of translates of $T(D)$. We get a similar pattern whenever $r=1$ and $n=2$.}
\label{fig:tiling}
\end{center}
\end{figure}
\end{exam}

\begin{rema}
Because our tiling is of $n$-dimensional space, it is difficult to present more complicated examples. However, in Section~\ref{sec:Lower}, we will show that we can take an $r$-dimensional or $(n-r)$-dimensional slice of our tiling and get many of the same results. This will allow us to present more interesting tilings of 2-dimensional space (see Figure~\ref{fig:tilings}). 
\end{rema}

\section{Constructing the Sandpile to Basis Multijections}\label{sec:multi}
In order to define our multijections, we will need $T(D)$ and an appropriate $\R^{n}$ direction vector.

\begin{defi}\label{def:shifting}
A \textit{shifting vector} $\w = (\w_1,\dots,\w_{n})$ of $D$ is a vector in $\R^{n}$ that is not in the span of a facet of $P(B)$ for any $B \in \B(D)$. 
\end{defi}

In Section~\ref{sec:shifting}, we will show that a choice of shifting vector is equivalent to a choice of \textit{chamber} from a certain \textit{hyperplane arrangement}. We use the same notation that we used in the previous section and $D$ is still an $r \times n$ standard representative matrix. 

It will sometimes be useful to split our shifting vector into two smaller vectors. Consider the vectors $w = (w_1,\dots,w_r) \in \R^r$ and $\widehat w = (\widehat w_1,\dots, \widehat w_{n-r})\in \R^{n-r}$. We write $(w,\widehat w)$ for their concatenation, which is an $\R^n$ vector.
\begin{lemma}\label{lem:splitshift}
$(w,\widehat w)$ is a shifting vector for $D$ if and only if for all $B \in \B(D)$, $w$ does not lie in the span of any facet of $P_1(B)$ and $\widehat w$ does not lie in the span of any facet of $P_2(B)$. 
\end{lemma}
\begin{proof}
By definition, a point is in $P(B)$ if and only if it is in $P_1(B)$ when restricted to the first $r$ coordinates and $P_2(B)$ when restricted to the last $(n-r)$ coordinates. The lemma follows from the fact that $z+ \epsilon w$ is $(z_1,\dots, z_r) + \epsilon (w_1,\dots, w_r)$ when restricted to the first $r$ coordinates and $(\widehat z_1,\dots,\widehat z_{n-r}) + \epsilon (\widehat w_1,\dots, \widehat w_{n-r})$ when restricted to the last $(n-r)$ coordinates.
\end{proof}

\begin{defi}\label{def:reps}
Let $\w = (w,\widehat w)$ be a shifting vector.
\begin{itemize}
\item For any $v \in \R^{n}$, $ v$ is a \textit{$\w$-representative} of $\tilde \Sand(D)$ if $ v+\epsilon \w \in T(D)$ for all sufficiently small $\epsilon>0$. If $ v+ \epsilon \w \in P(B)$, we say that $ v$ is \textit{$\w$-associated} with $B$. 

\item For any $z \in \Z^{n}$, $z$ is a \textit{$\w$-representative} of $\Sand(D)$ if $z+\epsilon \w \in T(D)$ for all sufficiently small $\epsilon>0$. If $z+ \epsilon \w \in P(B)$, we say that $z$ is \textit{$\w$-associated} with $B$. 

\item For any $v \in \R^r$ if $v+ \epsilon w \in P_1(B)$ for all sufficiently small $\epsilon>0$, we say that $v$ is \textit{$w$-associated} with $B$. 

\item For any $\widehat v \in \R^{n-r}$ if $\widehat v+ \epsilon \widehat w \in P_2(B)$ for all sufficiently small $\epsilon>0$, we say that $\widehat v$ is \textit{$\widehat w$-associated} with $B$. 
\end{itemize}
\end{defi}

\begin{lemma}\label{lem:splitassoc}
Suppose $\w = (w,\widehat w)$ is a shifting vector, $v \in \R^r$, and $\widehat v \in \R^{n-r}$. Then, $(v,\widehat v)$ is $\w$-associated with $B$ if and only if $v$ is $w$-associated with $B$ and $\widehat v$ is $\widehat w$-associated with $B$. 
\end{lemma}
\begin{proof}
For any $\epsilon>0$, the first $r$ entries of $(v,\widehat v) + \epsilon \w$ are given by $v + \epsilon w$, and the last $n-r$ entries are given by $\widehat v + \epsilon \widehat w$. The lemma follows from the fact that $P_1(B)$ is $P(B)$ restricted to its first $r$ coordinates while $P_2(B)$ is $P(B)$ restricted to its last $n-r$ coordinates.
\end{proof}

\begin{lemma}\label{lem:associated}
Each $\w$-representative of $\tilde \Sand(D)$ or $\Sand(D)$ is $\w$-associated with exactly one $B \in \B(D)$. 
\end{lemma}
\begin{proof}
Since $\w$-representatives of $\Sand(D)$ are also $\w$-representatives of $\tilde \Sand(D)$, it suffices to prove the result for $\tilde \Sand(D)$. Let $p$ be a $\w$-representative of $\tilde \Sand(D)$. Because $T(D) = \bigcup P(B)$, we know that $p + \epsilon \w \in P(B)$ for some $B \in \B(D)$. Since $\w$ is not in the span of any facet of $P(B)$, $p + \epsilon \w$ must be in the interior of $P(B)$. By Proposition~\ref{prop:nointersect}, this is true for a unique $B$.\end{proof}

\begin{prop}\label{prop:justone}
For any shifting vector $\w$, there is exactly one $\w$-representative in $\R^{n}$ for each equivalence class of $\tilde \Sand(D)$ and exactly one $\w$-representative in $\Z^{n}$ for each equivalence class of $\Sand(D)$.
\end{prop}

\begin{proof}
The second result is a direct corollary of the first (and could also be proven with an enumerative argument). By Corollary~\ref{cor:tiling}, every point $ p \in \R^{n}$ lies on some translation of $T(D)$ by an integer linear combination of the rows of $\D$. We can translate this point to a point on $T(D)$ without changing the equivalence class with respect to $\tilde \Sand(D)$. If $p$ maps to an interior point $p'$ of $T(D)$, then by Proposition~\ref{prop:lastDomino}, this is the unique point on $T(D)$ that is equivalent to $ p$. Furthermore, since $ p'$ is in the interior of $T(D)$, $ p'$ is always a $\w$-representative of $\tilde \Sand(D)$ regardless of $\w$.

If $ p$ maps to a boundary point of $T(D)$, then by Proposition~\ref{prop:lastDomino}, any point of $T(D)$ that is in the same $\tilde \Sand(D)$ equivalence class must also lie on the boundary of $T(D)$. Label these points as $\{ p_1,\dots, p_k\}$. We need to show that exactly one of these points is a $\w$-representative. 

By the condition that $\w$ is not in the span of any facet of $T(D)$, for all sufficiently small $\epsilon>0$, $ p_i + \epsilon \w$ must not lie on the boundary of $T(D)$ for any $i$. If $ p_i + \epsilon \w$ and $ p_j + \epsilon \w$ are both in $T(D)$ for $i \not= j$, then these are two distinct points in the interior of $T(D)$ that are equivalent as elements of $\tilde \Sand(D)$. This is impossible by Proposition~\ref{prop:lastDomino}.  

We have shown uniqueness, so we just need existence. Because $\w$ is not in the span of any facet of $T(D)$, we can choose $\epsilon>0$ so that all points between $p$ and $ p + \w\epsilon$ map to interior points of $T(D)$. Let $p'$ be the point mapped to by $ p + \w\epsilon$. Then, $p' - \epsilon \w$ must be equivalent to $p$ with respect to $\tilde \Sand(D)$. By our condition on $\epsilon$, we see that this point is a $\w$-representative. 
\end{proof}

\begin{prop}\label{prop:Ehrhart}
For any shifting vector $\w$, and for any $B \in \B(D)$, there are exactly $m(B)^2$ $\w$-representatives of $\Sand(D)$ that are $\w$-associated with $B$.
\end{prop}
To prove this result, we apply the following lemma from Ehrhart Theory:
\begin{lemma}[{\cite[Lemma 9.2]{Beck}}]\label{lem:Ehrhart}
For any integer matrix $M$, the number of integer points in the half-open fundamental parallelepiped $\Pi_\circ(M)$ is equal to its volume (the magnitude of $\det(M)$). 
\end{lemma}

\begin{proof}[Proof of Proposition~\ref{prop:Ehrhart}]
For some $B \in \B(D)$, let $\{\colB 1,\dots, \colB r\}$ be the columns of $D$ corresponding to $B$. Decompose $\w$ into the pair $(w,\widehat w)$ with $w \in \R^r$ and $\widehat w \in \R^{n-r}$. A point $v \in P_1(B)$ can be written as
\[v = \sum_{i=1}^{r} a_i\colB i,\]

\noindent with $0 \le a_i \le 1$ for all $i$. Because the $\colB i$ are linearly independent (otherwise $B$ would not be a basis), there is a unique way to write $w$ in the form:
\[ w = \sum_{i=1}^{r} b_i\colB i,\]

\noindent such that each $b_i \in \R$. By Lemma~\ref{lem:splitshift}, $w$ is not in the span of any facet of $P(B)$. This means that that $b_i \not= 0$ for all $i\in [r]$. For any $\epsilon \in \R$, we have:
\[v+\epsilon w = \sum_i^{n} (a_i + \epsilon b_i)\colB i.\]

From here, we see that $v$ is $w$-associated with $B$ if and only if $a_i \in (0,1]$ for $b_i < 0$ and $a_i \in [0,1)$ for $b_i > 0$. This region is the integer translation of a half-open fundamental parallelepiped with volume equal to the volume of $P_1(B)$. By an analogous line of reasoning, the points which are $\widehat w$-associated with $B$ form the integer translation of a half-open fundamental parallelepiped with volume equal to the volume of $P_2(B)$. It follows that the set of points that are $\w$-associated with $B$ is the direct product of these two regions: the integer translate of a half open parallelepiped with volume equal to the volume of $P(B)$.

By Lemma~\ref{lem:Ehrhart}, the number of integer points in this region is equal to this volume, and the integer translation does not change the number of integer points. Finally, by Lemma~\ref{lem:pipedvol}, the volume is $m(B)^2$, completing the proof. 
\end{proof}

We now define a function $\tilde f_{\w}$ from $\tilde \Sand(D) \to \B(D)$ given a shifting vector $\w$. For any $s \in \tilde \Sand(D)$, we first take the $\w$-representative $z$ of $s$ (which is unique by Proposition~\ref{prop:justone}). Then, we let $\tilde f_{\w}(s) = B$, where $B$ is the $\w$-associated basis of $z$ (which is unique by Lemma~\ref{lem:associated}). 

\begin{defi}\label{def:restrictedf}
$f_{\w}$ is $\tilde f_{\w}$ (as defined above) but with its domain restricted to $\Sand(D)$. 
\end{defi}

The following theorem is the main result of this paper.

\begin{theorem}\label{thm:Hellyeah}
For any $B\in \B(D)$, we have $|f_{\w}^{-1}(B)| = m(B)^2$.
\end{theorem}

\begin{proof}
We showed in Propositions~\ref{prop:justone} and~\ref{prop:nointersect} that $f_{\w}$ is a well-defined map from $\Sand(D)$ to $\B(D)$. The fact that $|f_{\w}^{-1}(B)| = m(B)^2$ is a corollary of Proposition~\ref{prop:Ehrhart}.
\end{proof}

\begin{exam}\label{ex:fave2}
Consider the matrix and associated tile from Example~\ref{ex:favorite}. One can show that $\w = (1,1,1)$ satisfies the requirements of a shifting vector. There are 14 different $\w$-representatives of $\Sand(D)$ given in the list below:
\begin{align*}
\{&(0,0,0),(0,0,-1),(1,0,-1),(1,1,-1),(2,1,-1),(2,2,-1),(0,0,-2),\\
&(1,0,-2),(1,1,-2),(2,1,-2),(2,2,-2),(0,0,-3),(1,1,-3),(2,2,-3)\}.
\end{align*}
Furthermore, we have:
\begin{align*}
f_\w^{-1}(\{1,2\}) = \{&(0,0,0)\}.\\
f_\w^{-1}(\{1,3\}) = \{&(1,0,-1),(2,1,-1),(1,0,-2),(2,1,-2)\}.\\
f_\w^{-1}(\{2,3\}) = \{&(0,0,-1),(1,1,-1),(2,2,-1),(0,0,-2),(1,1,-2),\\
&(2,2,-2),(0,0,-3),(1,1,-3),(2,2,-3)\}.\\
\end{align*}
\noindent where each $\w$-representative is shorthand for ``the equivalence class of $\Sand(D)$ containing this $\w$-representative''. We can confirm that $f_{\w}$ is a multijection by noting that:
\begin{align*}
|f_\w^{-1}(\{1,2\})| &= 1 = m(\{1,2\})^2.\\
|f_\w^{-1}(\{1,3\})| &= 4 = m(\{1,3\})^2.\\
|f_\w^{-1}(\{2,3\})| &= 9 = m(\{2,3\})^2.\\
\end{align*}

If we use a different shifting vector, some of our representatives may change. For example, for $\w' = (-1,2,-2)$, we have:
\begin{align*}
f_{\w'}^{-1}(\{1,2\}) = \{&(1,0,1)\}.\\
f_{\w'}^{-1}(\{1,3\}) = \{&(1,0,0),(2,1,0),(1,0,-1),(2,1,-1)\}.\\
f_{\w'}^{-1}(\{2,3\}) = \{&(3,2,0),(1,1,0),(2,2,0),(3,2,-1),(1,1,-1),\\
&(2,2,-1),(3,2,-2),(1,1,-2),(2,2,-2)\}.\\
\end{align*}

Note that interior points of $P(B)$ are always associated with $B$, but boundary points depend on the shifting vector. 
\end{exam}

\section{Lower-Dimensional Representatives}\label{sec:Lower}

In Section~\ref{sec:tiling}, we showed how to construct a tiling of $\R^{n}$ and then in Section~\ref{sec:multi}, we used this tiling to produce a set of representatives for $\Sand(D)$ (see Theorem~\ref{thm:Hellyeah}). In this section, we show how to use the tiling of $\R^{n}$ to produce a tiling of $\R^r$ or $\R^{n-r}$ that also (given a shifting vector) produces a set of representatives of $\Sand(D)$. The representatives associated with the tiling of $\R^r$ all have zero in their last $n-r$ entries while the representatives associated with the tiling of $\R^{n-r}$ all have zero in their first $r$ entries. However, even though the representatives of $\Sand(D)$ change, the multijection does not. 

One benefit of this alternate construction is that it is often easier to work in lower dimensional space. In particular, we are now able to produce a wide variety of tilings of $\R^2$ (see Figure~\ref{fig:tilings}). With our original map, all tilings of $\R^2$ were similar to the one given in Example~\ref{ex:2dim}.

The main tool we use in this section is the following lemma. 

\begin{lemma}\label{lem:lowerdim}
Let $D$ be the standard representative matrix
\[ D = \begin{pmatrix}
I_n &M \\
\end{pmatrix}\]

\noindent and let $z = (z_1,\dots, z_r,\widehat z_1,\dots, \widehat z_{n-r})^T\in \Z^{n}$. Then, $z$ is equivalent, with respect to $\Sand(D)$, to the vector whose first $r$ entries are given by
\[(z_1,\dots, z_r)^T + M^T(\widehat z_1,\dots,\widehat z_{n-r})^T,\]

\noindent and whose last $(n-r)$ entries are zero. 

$z$ is also equivalent, with respect to $\Sand(D)$, to the vector whose first $r$ entries are zero and whose last $(n-r)$ entries are given by
\[(\widehat z_1,\dots, \widehat z_{n-r})^T - M(z_1,\dots,z_r)^T.\]

\end{lemma}

\begin{proof}

The desired vectors are equal to 
\[z - \D^T(0,\dots, 0,\widehat z_1,\dots,\widehat z_{n-r})^T\]
\noindent and 
\[z - \D^T(z_1,\dots,z_r,0,\dots, 0)^T\] \noindent respectively. The lemma follows from Lemma~\ref{lem:whenequiv}.
\end{proof}

We also introduce two alternative integral bases for $\Sand(D)$ which will be useful when working in lower dimensions. 
\begin{prop}\label{prop:otherbases}
The rows of the following matrices are each integral bases for $\Ima_\Z(D^T) \oplus \ker_\Z(D)$:
\[\D' = \begin{pmatrix}
I_r & M\\
0 & \widehat D \widehat D^T\\
\end{pmatrix}
\hspace{4cm}
\D'' = \begin{pmatrix}
D D^T & 0\\
-M^T & I_{n-r}\\
\end{pmatrix}.\]
\end{prop}

\begin{proof}
Consider the following matrices:
\[U' = \begin{pmatrix}
I_r & 0\\
M^T & I_{n-r}\\
\end{pmatrix}
\hspace{4cm}
U'' = \begin{pmatrix}
I_r & -M\\
0 & I_{n-r}\\
\end{pmatrix}.\]
By the equalities $M^T M + I_{n-r} = \widehat D \widehat D^T$ and $M M^T + I_r =  D D^T$, we have $\D' = \D U'$ and $\D'' = \D U''$. Furthermore, $U'$ and $U''$ each have determinant 1 because they are triangular with ones along the diagonal. Thus, the proposition follows by the fact that the row lattice of a matrix doesn't change after multiplying by an integer matrix of determinant $1$. 
\end{proof}

Recall from Definition~\ref{def:p1p2p} that for any $B \in \B(D)$, we have parallelepipeds $P_1(B)$, $P_2(B)$, and $P(B)$, where $P(B)$ is the direct product of $P_1(B)$ and $P_2(B)$. Consider the vectors $w = (w_1,\dots,w_r) \in \R^r$, $\widehat w = (\widehat w_1,\dots, \widehat w_{n-r})\in \R^{n-r}$, and $\w = (w,\widehat w)$. Recall from Lemma~\ref{lem:splitshift} that $(w,\widehat w)$ is a shifting vector if $w$ is not in the span of any facet of $P_1(B)$ and $\widehat w$ is not in the span of any facet of $P_2(B)$. 

By a slight adjustment of Proposition~\ref{prop:Ehrhart}, one can show that there are $m(B)$ integer vectors $w$-associated with $P_1(B)$ and $m(B)$ integer vectors $\widehat w$-associated with $P_2(B)$. We now show how to construct an $r$-dimensional tile and an $(n-r)$-dimensional tile. For both constructions, we use a standard representative matrix $D$ and a shifting vector $(w,\widehat w) = (w_1,\dots, w_r, \widehat w_1,\dots, \widehat w_{n-r})$. 

\begin{defi}
\[ T'(D) = \bigcup_{B \in \B(D)} \left(\bigcup_{z\in \Z^{n-r} \text{ $\widehat w$-associated with } P_2(B)} \left(P_1(B) + M^T z^T\right)\right)\]
\end{defi}
 $T'(D)$ is made up of $m(B)$ parallelepipeds for each $B \in \B(D)$ and depends on $(\widehat w_1,\dots, \widehat w_{n-r})$ but not $(w_1,\dots, w_r)$. Figure~\ref{fig:T'D} gives an example of $T'(D)$.

\begin{defi}
\[ T''(D) = \bigcup_{B \in \B(D)} \left(\bigcup_{ z\in \Z^r \text{ $w$-associated with } P_1(B)} \left(P_2(B) - M z^T\right)\right)\]
\end{defi}

$T''(D)$ is made up of $m(B)$ parallelepipeds for each $B \in \B(D)$ and depends on $(w_1,\dots, w_r)$ but not $(\widehat w_1,\dots,\widehat w_{n-r})$. Figure~\ref{fig:T''D} gives an example of $T''(D)$.

The following theorem says that $T'(D)$ and $T''(D)$ have many similar properties to $T(D)$. This is the main result of this section.
\begin{theorem}
\hspace{2pt}

\begin{itemize}
    \item The parallelepipeds that make up $T'(D)$ only intersect at their boundaries. 
    \item The parallelepipeds that make up $T''(D)$ only intersect at their boundaries. 
    \item The set of translates $T'(D) + DD^T(z_1,\dots,z_r)^T$ for all $(z_1,\dots,z_r) \in \Z^r$ cover all of $\R^r$ and only intersect at their boundaries. 
    \item The set of translates $T''(D) + \widehat D\widehat D^T(\widehat z_1,\dots,\widehat z_{n-r})^T$ for all $(\widehat z_1,\dots,\widehat z_{n-r}) \in \Z^{n-r}$ cover all of $\R^{n-r}$ and only intersect at their boundaries. 
    \item For each $B\in \B(D)$, there are exactly $m(B)^2$ integer points $(z_1,\dots, z_n)$ of $T'(D)$ such that for all sufficiently small $\epsilon>0$, $(z_1,\dots, z_n) + \epsilon(w_1,\dots,w_n)$ is in one of the translates of $P_1(B)$ that make up $T'(D)$. 
    \item For each $B\in \B(D)$, there are exactly $m(B)^2$ integer points $(\widehat z_1,\dots, \widehat z_m)$ of $T''(D)$ such that for all sufficiently small $\epsilon>0$, $(\widehat z_1,\dots, \widehat z_m) + \epsilon(\widehat w_1,\dots,\widehat w_m)$ is in one of the translates of $P_2(B)$ that make up $T''(D)$. 
\end{itemize}
\end{theorem}

\begin{proof}
The general strategy for every part of this proof is to apply Lemma~\ref{lem:lowerdim} to results from Section~\ref{sec:multi} about $T(D)$. 

The first 2 parts follow from Proposition~\ref{prop:nointersect} and Lemma~\ref{lem:lowerdim}. 

For the next 2 parts, Proposition~\ref{prop:otherbases} implies that two $\R^{n}$ vectors that end with $(n-r)$ zeros are equivalent if and only if their difference when restricted to the first $r$ entries is in $\Ima_\Z(D D^T)$. Similarly, two $\R^{n}$ vectors that begin with $r$ zeros are equivalent if and only if their difference when restricted to the last $(n-r)$ entries is in $\Ima_\Z(\widehat D \widehat D^T)$. The results follow from this observation as well as Corollary~\ref{cor:tiling} and Lemma~\ref{lem:lowerdim}. 

Finally, for the last 2 parts, the integer points we obtain are exactly the $w$-representatives of $T(D)$ translated by Lemma~\ref{lem:lowerdim} so that either the first $r$ or last $(n-r)$ coordinates are 0. Thus, we can just apply Theorem~\ref{thm:Hellyeah}.
\end{proof}

\begin{exam} \label{ex:fave3}
Consider the matrix 

\[D = \begin{pmatrix}
1 & 0 & 3 \\
0 & 1 & 2\\
\end{pmatrix}
\text{ which is associated to the matrix }
\D = \begin{pmatrix}
1 & 0 & 3 \\
0 & 1 & 2\\
-3&-2 & 1\\
\end{pmatrix}.\]

In Example~\ref{ex:favorite}, we gave a perspective drawing for the 3-dimensional $T(D)$. In Example~\ref{ex:fave2}, we gave the set of $\w$-representatives when $\w = (1,1,1)$. Here, we will show how to construct $T'(D)$ and $T''(D)$ and find a set of $\w$-representatives for these lower-dimensional tiles. 

To construct $T'(D)$, we first look at $P_2(B)$ for each $B \in \B(D)$. Because $n-r=1$, these are intervals.

\begin{align*}
P_2(\{1,2\}) & = [0,1].\\ 
P_2(\{1,3\}) & = [-2,0].\\ 
P_2(\{2,3\}) & = [-3,0].\\ 
\end{align*}

Then, for each $B \in \B(D)$, we find the set of integer points that are mapped into $P_2(B)$ by the shifting vector $(1)$ (the last $(n-r)$ entries of $\w$). For $P_2(\{1,2\})$, this is $\{(0)\}$.  For $P_2(\{1,3\})$, this is $\{(-2),(-1)\}$.  For $P_2(\{2,3\})$, this is $\{(-3),(-2),(-1)\}$. Then, we multiply each of these by $(3,2)^T$ and shift $P_1(B)$ by these amounts. The resulting tile is given in Figure~\ref{fig:T'D}.

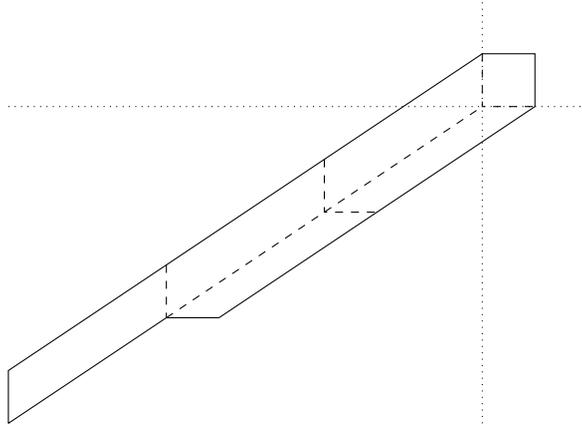
\begin{figure}
\begin{center}
\begin{tikzpicture}[scale = 0.7]
    
    \tikzstyle{every node} = [circle,fill,inner sep=1pt,minimum size = 1.5mm]
    \draw (0,0) -- (0,1) -- (9,7) -- (10,7) --(10,6) --(4,2) -- (3,2) --(0,0);
    
    \draw[dashed] (3,3) -- (3,2);
    \draw[dashed] (6,5) -- (6,4) -- (7,4);
    \draw[dashed] (9,7) -- (9,6) -- (10,6);
    \draw[dashed] (3,2) -- (9,6);
    \draw[dotted] (9,0) -- (9,8);
    \draw[dotted] (0,6) -- (11,6);

\end{tikzpicture}
\caption{This is $T'(D)$ for ${\widehat w} = (1)$. It is made up of 1 parallelogram of area 1 corresponding to $\{1,2\}$, 2 parallelograms of area 2 corresponding to $\{1,3\}$, and 3 parallelograms of area 3 corresponding to 
$\{2,3\}$. The dotted lines are the coordinate axes.}
\label{fig:T'D}
\end{center}
\end{figure}

Finally, to find a set of representatives for $\Sand(D)$, we take all of points $(z_1,z_2)\in \Z^2$ such that for all sufficiently small $\epsilon>0$, $(z_1,z_2) + \epsilon(1,1) \in T'(D)$ (where the shifting vector $(1,1)$ is from the first two elements of $\w$).
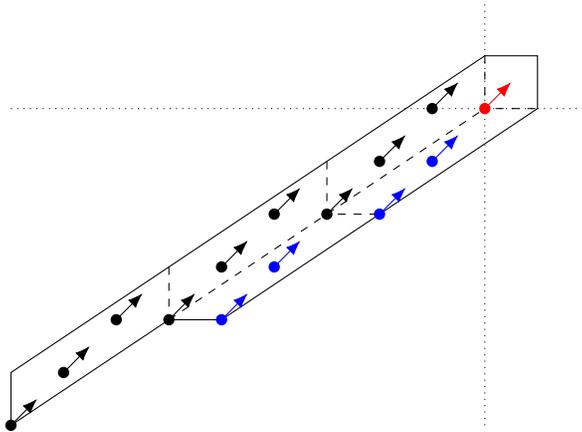
\begin{figure}
\begin{center}
\begin{tikzpicture}[scale = 0.7]
    
    \tikzstyle{every node} = [circle,fill,inner sep=1pt,minimum size = 1.5mm]
    \draw (0,0) -- (0,1) -- (9,7) -- (10,7) --(10,6) --(4,2) -- (3,2) --(0,0);
    
    \draw[dashed] (3,3) -- (3,2);
    \draw[dashed] (6,5) -- (6,4) -- (7,4);
    \draw[dashed] (9,7) -- (9,6) -- (10,6);
    \draw[dashed] (3,2) -- (9,6);
    \draw[dotted] (9,0) -- (9,8);
    \draw[dotted] (0,6) -- (11,6);
    
    \node (o) at (0,0){};
    \draw[-{Latex[length=2mm,width=1.5mm]}] (0,0)--(.5,.5);
    \node (o) at (1,1){};
    \draw[-{Latex[length=2mm,width=1.5mm]}] (1,1)--(1.5,1.5);
    \node (o) at (2,2){};
    \draw[-{Latex[length=2mm,width=1.5mm]}] (2,2)--(2.5,2.5);
    
    \node (o) at (3,2){};
    \draw[-{Latex[length=2mm,width=1.5mm]}] (3,2)--(3.5,2.5);
    \node (o) at (4,3){};
    \draw[-{Latex[length=2mm,width=1.5mm]}] (4,3)--(4.5,3.5);
    \node (o) at (5,4){};
    \draw[-{Latex[length=2mm,width=1.5mm]}] (5,4)--(5.5,4.5);
    
    \node (o) at (6,4){};
    \draw[-{Latex[length=2mm,width=1.5mm]}] (6,4)--(6.5,4.5);
    \node (o) at (7,5){};
    \draw[-{Latex[length=2mm,width=1.5mm]}] (7,5)--(7.5,5.5);
    \node (o) at (8,6){};
    \draw[-{Latex[length=2mm,width=1.5mm]}] (8,6)--(8.5,6.5);
    
    \node[color = blue] (o) at (4,2){};
    \draw[color = blue,-{Latex[length=2mm,width=1.5mm]}] (4,2)--(4.5,2.5);
    \node[color = blue] (o) at (5,3){};
    \draw[color = blue,-{Latex[length=2mm,width=1.5mm]}] (5,3)--(5.5,3.5);
    
    \node[color = blue] (o) at (7,4){};
    \draw[color = blue,-{Latex[length=2mm,width=1.5mm]}] (7,4)--(7.5,4.5);
    \node[color = blue] (o) at (8,5){};
    \draw[color = blue,-{Latex[length=2mm,width=1.5mm]}] (8,5)--(8.5,5.5);
    
    \node[color = red] (o) at (9,6){};
    \draw[color = red,-{Latex[length=2mm,width=1.5mm]}] (9,6)--(9.5,6.5);
\end{tikzpicture}
\caption{We show which integer points map into $T'(D)$ by the shifting vector $(1,1)$. The color of the point corresponds to which basis the point is mapped to. As expected from Theorem~\ref{thm:Hellyeah}, there is 1 point mapped to $\{1,2\}$, 4 points mapped to $\{1,3\}$, and 9 points mapped to $\{2,3\}$. If we append 0 to each of these points, we get a set of representatives for $\Sand(D)$.}
\label{fig:2dbij}
\end{center}
\end{figure}

Let $f_{\w}'$ be the map that sends $\Sand(D)$ to $\B(D)$ by mapping the lattice points in Figure~\ref{fig:2dbij} to bases associated to the parallelograms they are shifted into. We get the following set of representatives for $\Sand(D)$: 

\begin{align*}
f_{\w}'^{-1}(\{1,2\}) = \{&(0,0,0)\}.\\
f_{\w}'^{-1}(\{1,3\}) = \{&(-2,-2,0),(-1,-1,0),(-5,-4,0),(-4,-3,0)\}.\\
f_{\w}'^{-1}(\{2,3\}) = \{&(-3,-2,0),(-2,-1,0),(-1,0,0),(-6,-4,0),(-5,-3,0),\\
&(-4,-2,0),(-9,-6,0),(-8,-5,0),(-7,-4,0)\}.\\
\end{align*}

Note that these are the same representatives that we get if we apply the first part of Lemma~\ref{lem:lowerdim} to the representatives we obtained in Example~\ref{ex:fave2} with the same shifting vector. 

We can also find a set of representatives by using the tiling $T''(D)$ of $\R$. For each $B \in \B(D)$, we find the set of lattice points that are mapped into $P_1(B)$ by the shifting vector $(1,1)$. 

\begin{align*}
\text{For }P_1(\{1,2\}) &\text{ this is } \{(0,0)\}.\\ 
\text{For }P_1(\{1,3\}) &\text{ these are } \{(1,0),(2,1)\}.\\ 
\text{For }P_1(\{2,3\}) &\text{ these are } \{(0,0),(1,1),(2,2)\}.\\ 
\end{align*}

Then, we multiply each of these points by $(-3,-2)$ and shift $P_2(B)$ by these amounts. This gives the following collection of intervals that form $T''(D)$ (where the different intervals are separated by dashed lines):

\begin{center}
\begin{figure}[ h!]
\begin{tikzpicture}
    
    \tikzstyle{every node} = [circle,fill,inner sep=1pt,minimum size = 1.5mm]
    \draw (-13,0)--(1,0);
    
    \draw (-13,-.5) -- (-13,.5);
    \draw[dashed] (-10,-.5) -- (-10,.5);
    \draw[dashed] (-8,-.5) -- (-8,.5);   
    \draw[dashed] (-5,-.5) -- (-5,.5);  
    \draw[dashed] (-3,-.5) -- (-3,.5);
    \draw[dashed] (0,-.5) -- (0,.5);  
    \draw (1,-.5) -- (1,.5); 
    \tikzstyle{every node} = [fill = none,inner sep=1pt,minimum size = 1.5mm]
    \node (o) at (-13,-.8) {-13};
    \node (o) at (-10,-.8) {-10};
    \node (o) at (-8,-.8) {-8};
    \node (o) at (-5,-.8) {-5};
    \node (o) at (-3,-.8) {-3};
    \node (o) at (0,-.8) {0};
    \node (o) at (1,-.8) {1};

\end{tikzpicture}
\caption{This is $T''(D)$ for $w = (1,1)$. It is made up of 1 interval of length 1 corresponding to $\{1,2\}$, 2 intervals of length 2 corresponding to $\{1,3\}$, and 3 intervals of length 3 corresponding to 
$\{2,3\}$.} 
\label{fig:T''D}
\end{figure}
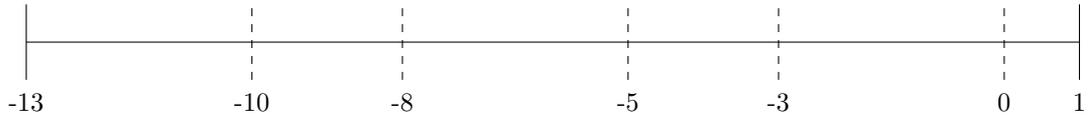
\end{center}

Finally, to find a set of representatives for $\Sand(D)$, we take all points $z$ such that for all sufficiently small $\epsilon>0$, $z + \epsilon(1) \in T''(D)$.
\begin{center}
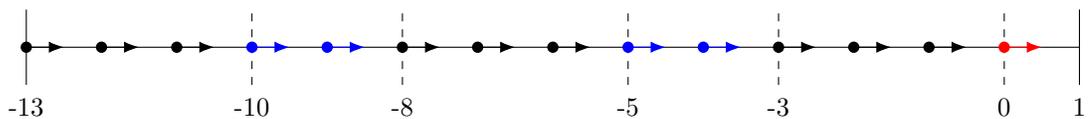
\begin{figure}[ h!]
\begin{tikzpicture}
    
    \tikzstyle{every node} = [fill = none,inner sep=1pt,minimum size = 1.5mm]
    \draw (-13,0)--(1,0);
    \draw (-13,-.5) -- (-13,.5);
    \draw[dashed] (-10,-.5) -- (-10,.5);
    \draw[dashed] (-8,-.5) -- (-8,.5);   
    \draw[dashed] (-5,-.5) -- (-5,.5);  
    \draw[dashed] (-3,-.5) -- (-3,.5);
    \draw[dashed] (0,-.5) -- (0,.5);  
    \draw (1,-.5) -- (1,.5); 
    \node (o) at (-13,-.8) {-13};
    \node (o) at (-10,-.8) {-10};
    \node (o) at (-8,-.8) {-8};
    \node (o) at (-5,-.8) {-5};
    \node (o) at (-3,-.8) {-3};
    \node (o) at (0,-.8) {0};
    \node (o) at (1,-.8) {1};
\tikzstyle{every node} = [circle,fill,inner sep=1pt,minimum size = 1.5mm]
\node (o) at (-13,0){};
    \draw[-{Latex[length=2mm,width=1.5mm]}] (-13,0)--(-12.5,0);
    \node (o) at (-12,0){};
    \draw[-{Latex[length=2mm,width=1.5mm]}] (-12,0)--(-11.5,0);
    \node (o) at (-11,0){};
    \draw[-{Latex[length=2mm,width=1.5mm]}] (-11,0)--(-10.5,0);
    
    \node (o) at (-8,0){};
    \draw[-{Latex[length=2mm,width=1.5mm]}] (-8,0)--(-7.5,0);
    \node (o) at (-7,0){};
    \draw[-{Latex[length=2mm,width=1.5mm]}] (-7,0)--(-6.5,0);
    \node (o) at (-6,0){};
    \draw[-{Latex[length=2mm,width=1.5mm]}] (-6,0)--(-5.5,0);
    
    \node (o) at (-3,0){};
    \draw[-{Latex[length=2mm,width=1.5mm]}] (-3,0)--(-2.5,0);
    \node (o) at (-2,0){};
    \draw[-{Latex[length=2mm,width=1.5mm]}] (-2,0)--(-1.5,0);
    \node (o) at (-1,0){};
    \draw[-{Latex[length=2mm,width=1.5mm]}] (-1,0)--(-.5,0);
    
    \node[color = blue] (o) at (-10,0){};
    \draw[color = blue,-{Latex[length=2mm,width=1.5mm]}] (-10,0)--(-9.5,0);
    \node[color = blue] (o) at (-9,0){};
    \draw[color = blue,-{Latex[length=2mm,width=1.5mm]}] (-9,0)--(-8.5,0);
    
    \node[color = blue] (o) at (-5,0){};
    \draw[color = blue,-{Latex[length=2mm,width=1.5mm]}] (-5,0)--(-4.5,0);
    \node[color = blue] (o) at (-4,0){};
    \draw[color = blue,-{Latex[length=2mm,width=1.5mm]}] (-4,0)--(-3.5,0);
    
    \node[color = red] (o) at (0,0){};
    \draw[color = red,-{Latex[length=2mm,width=1.5mm]}] (0,0)--(.5,0);

\end{tikzpicture}
\caption{We show which integer points map into $T''(D)$ by the shifting vector $(1)$. The color of the point corresponds to which basis the point is mapped to. As expected from Theorem~\ref{thm:Hellyeah}, there is 1 point mapped to $\{1,2\}$, there are 4 points mapped to $\{1,3\}$, and there are 9 points mapped to $\{2,3\}$. If we prepend $(0,0)$ to each of these points, we get a set of representatives for $\Sand(D)$.} 
\label{fig:1dbij}
\end{figure}
\end{center}

Let $f_{(w,\widehat w)}''$ be the map that sends $\Sand(D) \to \B(D)$ by mapping the lattice points in Figure~\ref{fig:1dbij} to bases associated to the intervals they are shifted into. We get the following set of representatives for $\Sand(D)$:

\begin{align*}
f_{\w}''^{-1}(\{1,2\}) = \{&(0,0,0)\}.\\
f_{\w}''^{-1}(\{1,3\}) = \{&(0,0,-10),(0,0,-9),(0,0,-5),(0,0,-4)\}.\\
f_{\w}''^{-1}(\{2,3\}) = \{&(0,0,-13),(0,0,-12),(0,0,-11),(0,0,-8),(0,0,-7),\\
&(0,0,-6),(0,0,-3),(0,0,-2),(0,0,-1)\}.\\
\end{align*}
\end{exam}

Note that these are the same representatives that we get as if we apply the second part of Lemma~\ref{lem:lowerdim} to the representatives we obtained in Example~\ref{ex:fave2} with the same shifting vector. 

Figure~\ref{fig:tilings} gives some examples of tiles in $\R^2$ computed using Sage. On the left is the tile with different colors indicating different bases and on the right is 9 copies of the tile to show how the tiling works.  

\begin{rema}
When $m(B) = 1$ for every $B \in \B(D)$, the tile $T'(D)$ consists of a single parallelepiped for each $B \in \B(D)$. It is possible to translate each of these parallelepipeds by vectors that are trivial with respect to $\Sand(D)$ and obtain the \textit{zonotope} formed by the columns of $D$. In~\cite{BBY}, the authors use this zonotope to construct bijections between $\B(D)$ and $\Sand(D)$ (when $m(B) = 1$ for all $B \in \B(D)$).  
\end{rema}

\begin{figure}
 \includegraphics[width=.48\linewidth]{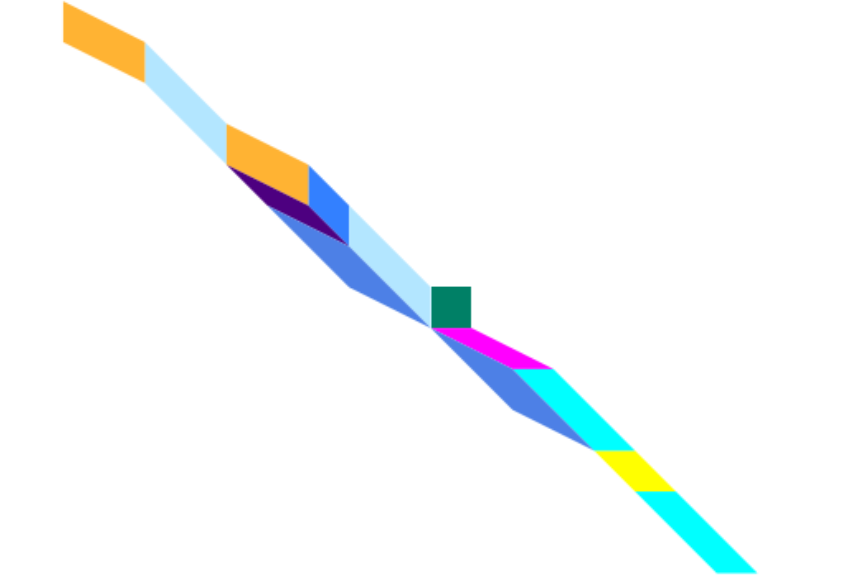}
\includegraphics[width=.48\linewidth]{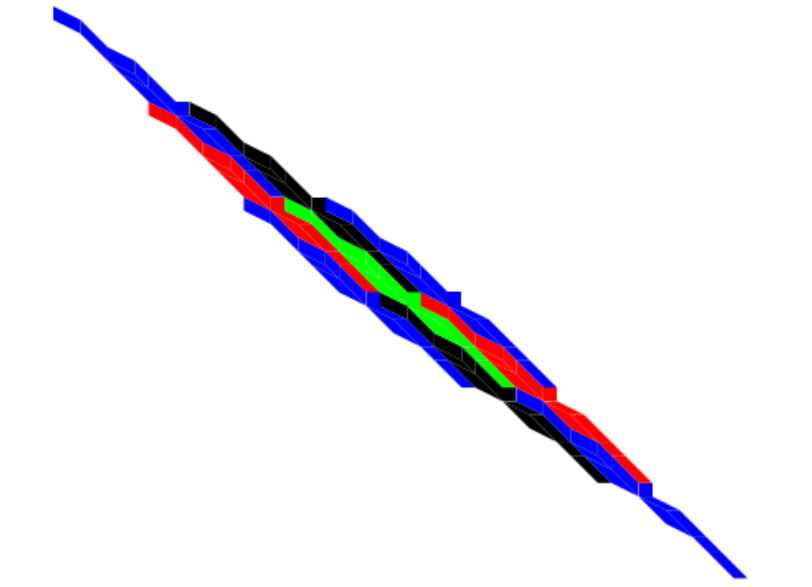}
\centering

\vspace{.5cm}

$D = \begin{pmatrix}
1 & 0 & -1 & -2 & 2 \\
0 & 1 & 1 & 2 & -1\\
\end{pmatrix} \hspace{1cm}\w = (1,1,1,1,1)$

\vspace{.5cm}

\centering
\includegraphics[width=.48\linewidth]{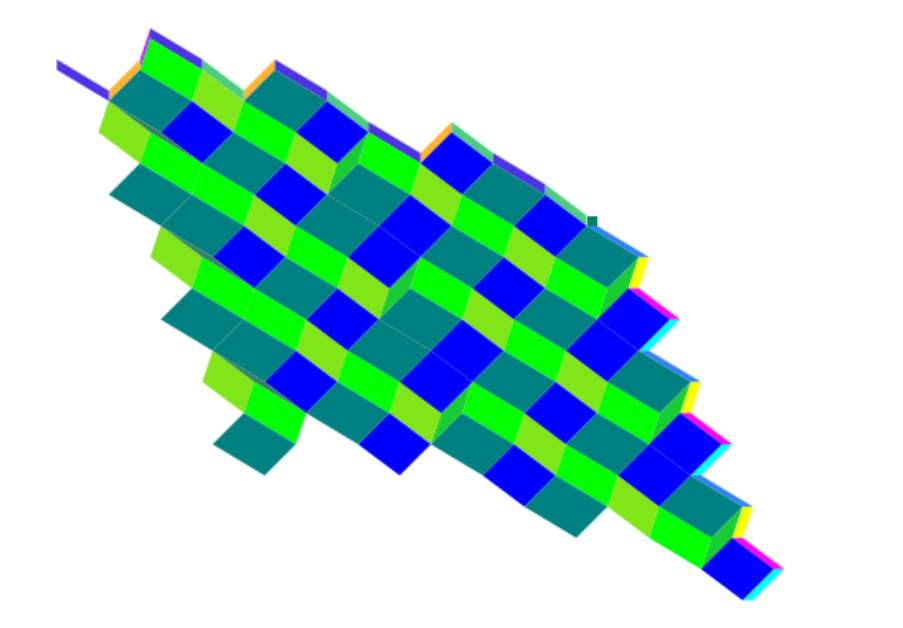}
\includegraphics[width=.48\linewidth]{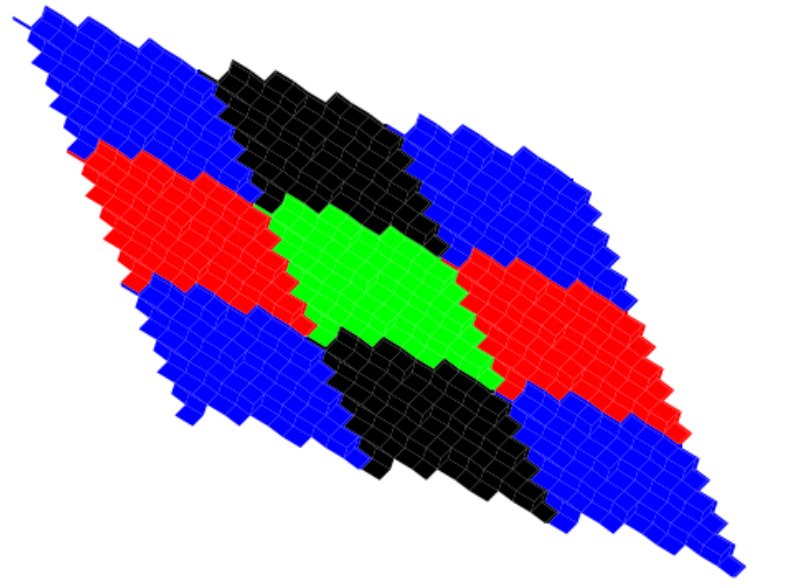}
\linebreak
\centering

\vspace{.5cm}

$D = \begin{pmatrix}
1 & 0 & 1 & 3 & -4 & 5 \\
0 & 1 & 3 & 3 & 3 & -3\\
\end{pmatrix}  \hspace{1cm}\w = (4,1,5,2,3,2)$

\vspace{.5cm}

\centering
\includegraphics[width=.48\linewidth]{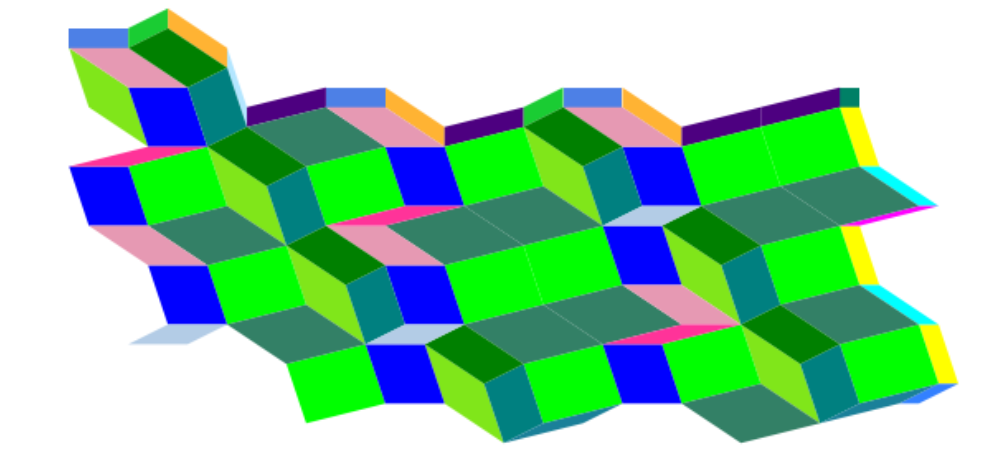}
\includegraphics[width=.48\linewidth]{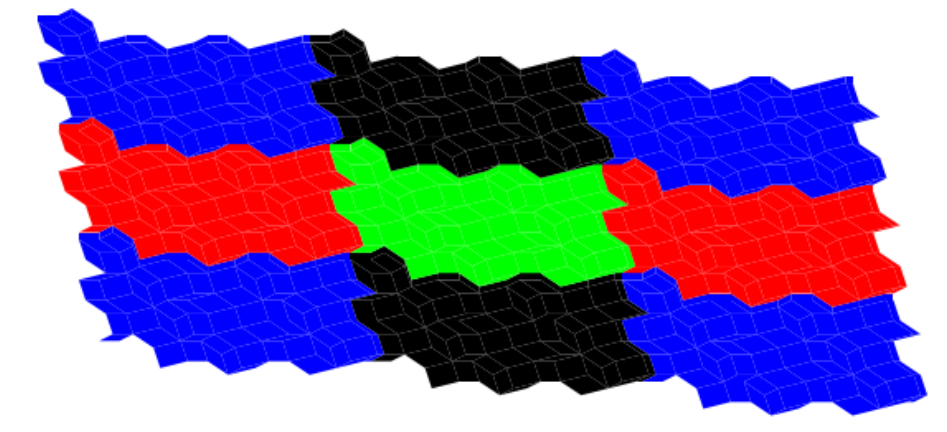}
\linebreak
\centering

\vspace{.5cm}

$D = \begin{pmatrix}
1 & 0 & 1 & 3 & -4 & 3 & 2 \\
0 & 1 & -3 & -2 & -1 & 0 & 1\\
\end{pmatrix}  \hspace{1cm}\w = (1,1,5,4,3,2,2)$

\vspace{.5cm}

\caption{Above are 3 examples of tiles that we obtain by applying Lemma~\ref{lem:lowerdim} to the higher-dimensional tiling from Section~\ref{sec:tiling}. }
\label{fig:tilings}
\end{figure}

\section{Shifting vectors and hyperplane arrangements}\label{sec:shifting}

In this section, we associate classes of shifting vectors producing the same multijection with \textit{chambers} of a \textit{hyperplane arrangement}. We also show that for a shifting vector $\w$, each basis is $\w$-associated with a unique \textit{corner point}. In the next section, we will show that each corner point is equivalent with respect to $\Sand(D)$ to a $\{0,1\}^n$ vector.

Recall that a \textit{standard representative matrix} is a matrix of the form
\[ D = \begin{pmatrix}
I_r &M \\
\end{pmatrix},\]
\noindent for some $r \times (n-r)$ integer matrix $M$. As in previous sections, we let $\widehat D$ be the $(n-r) \times n$ matrix
\[\widehat D = \begin{pmatrix}
-M^T & I_{n-r} \\
\end{pmatrix},\]
\noindent and $\D$ be the $n\times n$ matrix
\[ \D = \begin{pmatrix}
D \\
\widehat D \\
\end{pmatrix} =\begin{pmatrix}
I_r & M \\
-M^T & I_{n-r} \\
\end{pmatrix}.\]

Let $D$ be a rank $r$ standard representative matrix and let $\B(D)$ be its set of bases. For $B \in \B(D)$, we write $\CB$ as shorthand for the columns of $D$ which correspond to the indices of $B$. Similarly, $\B(\widehat D)$ is the set of bases of $\what D$, and for $\what B \in \B(\what D)$, we write $\hCB$ as shorthand for the columns of $\widehat D$ which correspond to the indices of $\what B$. 

\begin{defi}
For a positive integer $k$, a \textit{central hyperplane} in $\R^k$ is a $(k-1)$-dimensional linear subspace of $\R^k$. An \textit{affine hyperplane} is a translated central hyperplane. We use the blanket term \textit{hyperplane} when we allow both central and affine hyperplanes. For a hyperplane $H$ and vector $v \in \R^k$, we define the affine hyperplane
\[H+v = \{p+v\mid p \in H\}.\]
A \textit{hyperplane arrangement} $\mathcal H$ is a collection of hyperplanes in $\R^k$. A \textit{chamber} of a hyperplane arrangement $\mathcal H$ is a connected component of
\[\R^k \setminus \left(\bigcup_{H \in \mathcal H} H\right ) .\]
\end{defi}

Let $S$ be a subset of $[n]$ and $\CS$ be the corresponding columns of $D$. We write $\spn(\CS)$ for the subspace of $\R^r$ generated over $\R$ by the vectors in $\CS$. Let $\rk(S)$ be the dimension of the space $\spn(\CS)$. We will be primarily working with the case where $\rk(S) = r-1$, in which case $\spn(\CS)$ is a central hyperplane in $\R^r$. 

\begin{defi}
$\mathcal H(D)$ is the hyperplane arrangement defined by:
\[ \mathcal H(D) = \bigcup_{\{S \subset [n] ~\mid~ \text{rk}(S) = r-1\}} \spn(x_S).\]
\end{defi}

The arrangement $\mathcal H(\what D)$ is defined analogously (but with $r-1$ replaced by $n-r-1$). For each $B \in \B(D)$, recall the parallelepiped $P_1(B)$ from Definition~\ref{def:p1p2p} (i.e. the fundamental parallelepiped of $D$ restricted to columns in $B$). 

\begin{lemma}\label{lem:p1hyp} $P_1(B)$ is the region bounded by the following set of $2r$ hyperplanes:
\[\{\spn(\CB \setminus x) \mid x \in \CB\} \cup \{\spn(\CB \setminus x) + x  \mid x \in \CB\}.\]
\end{lemma}
\begin{proof}
Let $\CB = \{\colB 1,\dots, \colB r\}$. Since $B$ is a basis, we can write any point $p \in \R^n$ uniquely in the form:
\[p = \sum_{i=1}^r a_i \colB i.\]

For each $\colB i\in \CB$, $\spn(\CB \setminus \colB i)$ and $\spn(\CB \setminus \colB i)+\colB i$ are parallel hyperplanes. Furthermore, for $\colB j \in \CB$ with $j \not= i$, the vector $\colB j$ is parallel to both hyperplanes. This means that we can determine whether or not $p$ is between $\spn(\CB \setminus \colB i)$ and $\spn(\CB \setminus \colB i)+\colB i$ while only considering $a_i$. If $a_i = 0$, $p$ lies on the first hyperplane, while if $a_i = 1$, $p$ lies on the second hyperplane. It follows that $p$ lies between the two hyperplanes precisely when $0 \le a_i \le 1$. Since this is true for every $i$, we conclude that $p$ lies in the region bounded by the hyperplanes precisely when $0 \le a_i \le 1$ for all $i \in [r]$. This is the same condition that determines whether or not $p \in P_1(B)$. 
\end{proof}

\begin{defi}
Fix some $B \in \B(D)$. Let $\phi_B$ be the map from $P_1(B) \times \CB$ to $\{0,1,2\}$ defined in the following way:
\[\phi_B(p,x) = \begin{cases}
    1 & \text{ if } p \in \spn(\CB\setminus x),\\
    2 &\text{ if } p \in (\spn(\CB\setminus x)+x),\\
    0 &\text{ if } p \notin \spn(\CB\setminus x) \cup (\spn(\CB\setminus x)+x).
\end{cases}\]
\end{defi}

This map is well-defined since a point cannot lie in two parallel hyperplanes. 

\begin{defi}\label{def:corner}
A \textit{corner point} of $P_1(B)$ is a $p \in P_1(B)$ such that for every $\colB i \in \CB$, we have $\phi_B(p,\colB i) \not= 0$.
\end{defi}

\begin{lemma}
For every $B \in \B(D)$, there are exactly $2^r$ corner points of $P_1(B)$ (one for each element of $\{1,2\}^r$) and they are all in $\Z^r$. 
\end{lemma}
\begin{proof}
For each $\xi \in \{1,2\}^r$ there is exactly one point $p$ such that for every $i \in [r]$, $\phi_B(p,\colB i)$ is the $i^{th}$ entry of $\xi$. This point is explicitly given by
\[p = \sum_{\{i~\mid ~\text{ the $i^{th}$ entry of $\xi$ is $2$}\}} \colB i.\]
Since each $\colB i$ is in $\Z^r$, the point $p$ is also in $\Z^r$. 
\end{proof}

We recover analogous results and  definitions as above when we replace $D$ with $\widehat D$, $\B(D)$ with $\B(\widehat D)$, $r$ with $n-r$, and $P_1(B)$ with $P_2(B)$. In particular, we get a hyperplane arrangement $\mathcal H(\widehat D)$ whose hyperplanes are spanned by sets of $n-r-1$ columns of $\widehat D$. Corner points of $P_2(B)$ are defined analogously to corner points of $P_1(B)$.

\begin{defi}
A \textit{corner point} of $P(B)$ is a $\Z^n$ vector whose first $r$ entries form a corner point of $P_1(B)$ and whose last $n-r$ entries form a corner point of $P_2(B)$. 
\end{defi}

Consider the vectors $w = (w_1,\dots,w_r) \in \R^r$ and $\widehat w = (\widehat w_1,\dots, \widehat w_{n-r})\in \R^{n-r}$. We write $(w,\widehat w)$ for their concatenation, which is an $\R^n$ vector. 

Recall from Definition~\ref{def:shifting} that $(w,\widehat w)$ is a \textit{shifting vector} if and only if for all $B \in \B(D)$, $(w,\widehat w)$ is not in the span of any facet of $P(B)$. By Lemma~\ref{lem:splitshift}, this is equivalent to the condition that for all $B \in \B(D)$, $w$ is not in the span of any facet of $P_1(B)$ and $\widehat w$ is not in the span of any facet of $P_2([n] \setminus B)$.

\begin{lemma}\label{lem:shiftnothyp}
$(w,\widehat w)$ is a shifting vector if and only if $w$ does not lie on any $H \in \mathcal H(D)$ and $\widehat w$ does not lie on any $\widehat H \in \mathcal H(\widehat D)$. 
\end{lemma}
\begin{proof}
By Lemma~\ref{lem:p1hyp}, each facet of $P_1(B)$ is contained in the hyperplane $\spn(\CB \setminus x)$ for some $x \in \CB$ (or its translation). Furthermore, every hyperplane of this form is the span of a facet of $P_1(B)$. It follows that $w$ satisfies the conditions for a shifting vector if and only if $w$ does not lie in any of the hyperplanes:
\[\bigcup_{B \in \B(D)}\left (\bigcup_{x \in \CB} \spn(\CB \setminus x)\right).\]

We claim that these are exactly the hyperplanes that make up $\mathcal H(D)$. This is true because $\CB \setminus x$ is always a set of $r-1$ linearly independent columns of $D$ and every set of $r-1$ linearly independent columns of $D$ can be extended to form a basis. It is analogous to show that the spans of the facets of $P_2([n] \setminus \widehat B)$ over all $\widehat B \in \B(\widehat D)$ correspond to the hyperplanes in $\mathcal H(\widehat D)$. The lemma follows. 
\end{proof}

From Lemma~\ref{lem:shiftnothyp}, we see that if $(w,\widehat w)$ is a shifting vector, $w$ must lie in a chamber of $\mathcal H(D)$ and $\widehat w$ must lie in a chamber of $\mathcal H(\widehat D)$. Let $B \in \B(D)$, $z = (z_1,\dots,z^r) \in \Z^r$, and $\widehat z = (\widehat z_1,\dots, \widehat z_{n-r})\in \Z^{n-r}$. Recall from Definition~\ref{def:reps} that for any $v \in \R^r$ (resp. $\widehat v \in \R^{n-r}$), $v$ (resp. $\widehat v$) is \textit{$w$-associated} with $B$ if $v+ \epsilon w \in P_1(B)$ (resp. $\widehat v + \epsilon \widehat w \in P_2(B)$) for all sufficiently small $\epsilon>0$.

\begin{prop}\label{prop:onecorner}
For any shifting vector $(w,\widehat w)$ and any choice of $B \in \B(D)$, there is a unique corner point of $P_1(B)$ that is $w$-associated with $B$, a unique corner point of $P_2(B)$ that is $\widehat w$-associated with $B$, and a unique corner point of $P(B)$ that is $\w$-associated with $B$. 
\end{prop}
\begin{proof}
The proof of this Proposition is similar to the proof of Proposition~\ref{prop:Ehrhart}. 

Let $\{\colB 1,\dots, \colB r\}$ be the columns of $D$ corresponding to $B$. An integer point $z \in P(B)$ can be written as
\[z = \sum_{i=1}^{r} a_i\colB i,\]

\noindent with $0 \le a_i \le 1$ for all $i$. Because the $\colB i$ are linearly independent (otherwise $B$ would not be a basis), there is a unique way to write $w$ in the form:
\[ w = \sum_{i=1}^{r} b_i\colB i,\]

\noindent for $b_i \in \R$. Because $(w,\widehat w)$ is a shifting vector, $b_i \not= 0$ for all $i\in [r]$. For any $\epsilon \in \R$, we have:
\[z+\epsilon w = \sum_i^{n} (a_i + \epsilon b_i)\colB i.\]

From here, we see that $z$ is $w$-associated with $B$ if and only if $0 < a_i \le 1$ for $b_i < 0$ and $0 \le a_1 < 1$ for $b_i > 0$. Furthermore, $z$ can only be a corner point if $a_i \in \{0,1\}$ for all $i$. Thus, the unique corner point $w$-associated with $B$ is given by taking $a_i = 0$ for $b_i>0$ and $a_i = 1$ for $b_i < 0$. 

The proof is analogous for $P_2(B)$ and $\widehat w$. From here, the fact that $P(B)$ has a unique $\w$-associated corner point follows from Lemma~\ref{lem:splitassoc}.
\end{proof}


\begin{defi}\label{def:wequiv}
Two shifting vectors $(w,\widehat w)$ and $(w',\widehat w')$ are \textit{equivalent} if $w$ and $w'$ lie in the same chamber of $\mathcal H(D)$ and $\widehat w$ and $\widehat w'$ lie in the same chamber of $\mathcal H(\widehat D)$. 
\end{defi}

\begin{prop}\label{prop:tfaeshifthyp}
Let $\w = (w,\widehat w)$ and $\w' = (w',\widehat w')$ be shifting vectors. The following are equivalent:
\begin{enumerate}
    \item $\w$ and $\w'$ are equivalent (in the sense of Definition~\ref{def:wequiv}).
    \item For every $B \in \B(D)$, the lattice points $\w$-associated to $B$ and the lattice points $\w'$-associated to $B$ coincide.
    \item The set of $\w$-representatives and the set of $\w'$-representatives coincide.
\end{enumerate}
\end{prop}
\begin{proof}
Let $B \in \B(D)$ and $z= (z_1,\dots, z_r) \in \Z^r \cap P_1(B)$. By Lemma~\ref{lem:p1hyp}, $(z_1,\dots, z_r) + \epsilon (w_1,\dots, w_r) \in P_1(B)$ if this sum is between $\spn(\CB \setminus x)$ and $\spn(\CB \setminus x) + x$ for every $x \in \CB$. This holds for sufficiently small $\epsilon>0$ precisely when the following two conditions hold:
\begin{itemize}
    \item For all $x \in \CB$ with $\phi_B(z,x) = 1$, $w$ is on the same side of $\spn(\CB\setminus x)$ as $x$. 
    \item For all $x \in \CB$ with $\phi_B(z,x) = 2$, $w$ is on the opposite side of $\spn(\CB\setminus x)$ as $x$. 
\end{itemize}
These conditions only depend on the chamber of $w$. We can make an analogous statement about when $(\widehat z_1,\dots, \widehat z_{n-r}) + \epsilon (\widehat w_1,\dots, \widehat w_{n-r}) \in P_2(B)$. It follows that 1 implies 2. The fact that 2 implies 3 is immediate. 

Lastly, we need to show that 3 implies 1. We prove the contrapositive. Suppose that $(w,\widehat w)$ and $(w',\widehat w')$ are not equivalent. Without loss of generality, we can assume $w$ and $w'$ are not in the same chamber of $\mathcal H(D)$ (otherwise we could make an analogous argument regarding $\what w$ and $\what w'$. This means that for some $S \subset [n]$ with $\rk(S) = r-1$, $w$ and $w'$ are on opposite sides of $\spn(x_S)$. Choose a $B_1 \in \B(D)$ such that $B_1 \cap S = r-1$ (this is always possible by choosing a maximal independent subset of $S$ and extending to a basis). Then, since $\spn(x_S)$ is the span of a facet of $P_1(B_1)$, $w$ and $w'$ do not associate the same corner point with $B_1$. This also means that $\w$ and $\w'$ do not associate the same corner point with $B_1$. 

let $z_1$ be the corner point of $P(B_1)$ that is $\w$-associated with $B_1$. If $z_1$ is not a $\w'$-representative, we are done. Otherwise, $z_1$ is $\w'$-associated with some basis $B_2 \not= B_1$. Let $z_2$ be the corner point of $P(B_2)$ that is $\w$-associated with $B_2$. Again, if $z_2$ is not a $\w'$-representative, we are done. Otherwise, $z_2$ is $\w'$-associated with some $B_3$. 

By repeating this process, we get a sequence $(z_1,z_2,\dots)$ of points in $\Z^n$ and a sequence $(B_1,B_2,\dots)$ of bases in $\B(D)$ such that for every $i$, $z_i$ is the $\w$-associated corner point of $B_i$ and is $\w'$-associated with $B_{i+1}$. 

If this process terminates, we reach a $\w$-representative that is not a $\w'$-representative and the proposition follows. If the process does not terminate, we get an infinite sequence of $B_i$. Since $\B(D)$ is finite, this sequence must repeat. We will show that this is impossible.

Notice that for each $i$, $z_i$ must be on the boundary of both $P(B_i)$ and $P(B_{i+1})$. Using the ideas from Proposition~\ref{prop:onecorner}, the corner point of $P(B)$ is the minimum value of $p \cdot \w$ over all points of $P(B)$. This means that $z_i \cdot \w$ must be an increasing sequence, and it is impossible for the sequence of $z_i$'s to repeat. 
\end{proof}

Proposition~\ref{prop:tfaeshifthyp} implies that if $(w,\widehat w)$ and $(w',\widehat w')$ are equivalent, then $f_{(w,\widehat w)}$ and $f_{(w',\widehat w')}$ are equivalent multijections. However, the converse does not quite hold in general. For example, consider the matrix:
\[\D = \begin{pmatrix} 1 & 1 \\ -1 & 1 \end{pmatrix}.\]
The shifting vectors $(w,\widehat w) = (1,1)$ and $(w',\widehat w') = (-1,-1)$ are not equivalent, but they do induce the same sandpile multijection. 

From Proposition~\ref{prop:tfaeshifthyp}, the number of classes of equivalent shifting vectors is equal to the number of chambers of $\mathcal H(D)$ multiplied by the number of chambers of $\mathcal H(\widehat D)$. This quantity is known to depend only on the \textit{oriented matroid} represented by $D$ (not depending on basis multiplicities) and can be calculated using Zaslavsky's Theorem (see~\cite{Zaslavsky}). 
\begin{rema}
Each equivalence class of shifting vectors can be associated with a choice of \textit{acyclic circuit and cocircuit signatures} (see~\cite[Section 8.2]{mythesis}). These are what Backman, Baker, and Yuen use to define their bijections when restricting to regular matroids. 
\end{rema}

\section{Corner Points as $\{0,1\}^n$ Vectors}\label{sec:cornp}

Let $D$ be a standard representative matrix and $(w,\widehat w)$ be a shifting vector. We showed in Proposition~\ref{prop:onecorner} that there is a unique \textit{corner point} of $P_1(B)$ that is \textit{$w$-associated} with $B$ and a unique corner point of $P_2(B)$ that is \textit{$\widehat w$-associated with $B$}. Using ideas from the proof of Proposition~\ref{prop:onecorner}, we can explicitly construct this corner point. We can also construct a $\{0,1\}^n$ vector that is in the same sandpile group equivalence class as this corner point. 

Let $\CB = \{\colB 1,\dots, \colB r\}$ be the columns of $D$ corresponding to $B$ and $\hCB = \{\hcolB 1, \dots, \hcolB {n-r}\}$ be the columns of $\widehat D$ corresponding to $\widehat B = E \setminus B$. Because $B$ is a basis, $\begin{pmatrix} \colB 1 \cdots \colB r \end{pmatrix}$ and $\begin{pmatrix} \hcolB 1 \cdots \hcolB {n-r} \end{pmatrix}$ are both invertible matrices. It follows that there is a unique vector $a = (a_{k_1},\dots,a_{k_r})$ such that $\begin{pmatrix} \colB 1\cdots \colB r \end{pmatrix} a^T = w^T$ and a unique vector $\widehat a = (\widehat a_{\widehat k_1}, \dots, \widehat a_{\widehat k_{n-r}})$ such that $\begin{pmatrix} \hcolB 1\hdots \hcolB {n-r} \end{pmatrix} \widehat a^T = \widehat w^T$. The shifting vector condition tells us that for all $i\in [n]$, $a_i \not= 0$ and $\widehat a_i \not= 0$. 

Let $v\in\Z^r$ be the sum,
\[\sum_{\{i \in [r]~\mid~a_{k_i} >0\}} \colB i,\]
and let $\widehat v \in \Z^{n-r}$ be the sum,
\[\sum_{\{i \in [n-r]~\mid~\hat a_{\hat k_i} >0\}} \hcolB i.\]
Let $\tilde p_{(B,(w,\widehat w))}$ be the concatenation $(v,\widehat v) \in \Z^n$. The following lemma is immediate from the proof of Proposition~\ref{prop:onecorner}.

\begin{lemma}
$\tilde p_{(B,(w,\widehat w))}$ is the unique corner point of $P(B)$ that is $(w,\widehat w)$-associated with $B$. 
\end{lemma}

We also construct the following point in $\{0,1\}^n$ which we call $p_{(B,(w,\what w))}$. 
\[\text{The } i^{th} \text{ entry of }p_{(B,(w,\what w))}= \begin{cases}   0 & \text{if } i \in B \text{ and } a_i > 0,\\
& \text{or if } i \not\in B \text{ and } \widehat a_i > 0.\\
1 & \text{if } i \in B \text{ and } a_i < 0,\\
& \text{or if } i \not\in B \text{ and } \widehat a_i < 0.
\end{cases}\]

\begin{prop}\label{prop:cornerzeroone}
$p_{(B,(w,\what w))}$ and $\tilde p_{(B,(w,\what w))}$ are in the same sandpile equivalence class. 
\end{prop}
\begin{proof}
In the construction of $\tilde p_{(B,(w,\widehat w))}$, when we add $\colB i$ for $i \le r$, this adds $1$ to the $i^{th}$ coordinate. When we add $\colB i$ for $i > r$, we can subsequently add the $i^{th}$ row of $\D$ without changing the equivalence class of $\Sand(D)$. The net effect is that we add $1$ to the $i^{th}$ coordinate. Similarly, when we add $\hcolB i$ for $i > r$, this adds $1$ to the $i^{th}$ coordinate. When we add $\hcolB i$ for $i\le r$, we can subsequently add the $i^{th}$ row of $\D$ and the net effect is that we add $1$ to the $i^{th}$ coordinate. This procedure adds rows of $\D$ to $\tilde p_{(B,(w,\what w))}$ and produces the point $p_{(B,(w,\widehat w))}$.
\end{proof}

\begin{coro}
Let $B \in \B(D)$ and $z$ be a $\{0,1\}^n$ vector. There is a choice of shifting vector $\w$ such that the corner point of $P(B)$ that is $\w$-associated to $B$ is equivalent to $z$ with respect to $\Sand(D)$. 
\end{coro}
\begin{proof}
We can choose almost any $a\in \R^r$ and $\hat a \in \R^{n-r}$ that satisfy the correct sign pattern such that $p_{(B,\w)} = z$. The only restriction is that we need to make sure that $\w$ is not in the span of any facet, but these exceptions form a set of measure $0$. We can always convert to a shifting vector without affecting the sign pattern of $a$ or $\hat a$ (since we already require these vectors to contain no zeros). 
\end{proof}

\begin{rema}
Consider the $r$-dimensional \textit{zonotope} $Z_D$ formed by the \textit{Minkowski sum} of the columns of $D$. Every $\{0,1\}^n$ vector $z$ is associated with the vertex $D\cdot z^T$. It follows that for every $B \in \mathcal B(D)$, the point $D\cdot p_{(B,(w,\what w))}^T$ is inside of $Z_D$. In \cite{BBY}, the authors use a \textit{zonotopal tiling} argument to show that each $p_{(B,(w,\what w))}$ is in a different equivalence class of $\Sand(D)$. Proposition~\ref{prop:cornerzeroone} (along with results from Section~\ref{sec:multi}) gives an alternative proof of this fact. 
\end{rema}

\section{Further Questions}\label{sec:quest}

The main purpose of our map was to associate each equivalence class of the sandpile group to a basis. However, in constructing this map, we also give a representative for each equivalence class. In particular, this is the set of $\w$-representatives.

\begin{question}
What are some properties of the $\w$-representatives that we get from different choices of distinguished basis or shifting vector? Are they generalizations of any known sets of representatives of the graphical sandpile group (such as superstable or critical configurations)? What about the lower dimensional representatives from Section~\ref{sec:Lower}?
\end{question}

In~\cite[Chapter 9]{mythesis}, the multijections in this paper are generalized to a larger class of objects. However, the sandpile group must be replaced with its \textit{Pontryagin dual}. Note that the Pontryagin dual of the cokernel of a lattice generated by the rows of a matrix is the cokernel of the lattice generated by its columns.  In the case of standard representative matrices, the sandpile group is canonically isomorphic to its Pontryagin dual. In general, the the groups are isomorphic, but these isomorphisms are non-canonical. 

\begin{question}\label{q:canonlattice}
What are some properties of this Pontryagin dual sandpile group and why does it allow for more natural multijections?
\end{question}

In this paper, we focus on standard representative matrices, but the ideas can naturally be restated in terms of \textit{representable arithmetic matroids} (more precisely \textit{orientable arithmetic matroids with the strong GCD property}) which is the framework used in~\cite{mythesis}. However, it is essential for our definition that these matroids are \textit{representable}.
\begin{question}
Is there a reasonable way to define the sandpile group of some class of non-representable matroids?
\end{question}

\longthanks{
The author would like to thank Matthew Baker, Seth Chaiken, Galen Dorpalen-Barry, Caroline Klivans, Giovanni Inchiostro, Chi Ho Yuen, and the anonymous reviewers for useful conversation, comments, and suggestions. }

\bibliographystyle{amsplain-ac}
\bibliography{sample}
\end{document}